\documentclass[11pt,reqno]{amsart}
\usepackage{amsfonts,amssymb,amsmath,amsthm}
\usepackage[margin=1in]{geometry}
\usepackage[hidelinks]{hyperref}
\usepackage{enumerate}
\usepackage{microtype}
\usepackage{mathtools}
\usepackage{cite}
\usepackage{tikz-cd}

\newcommand\id{\operatorname{id}}
\newcommand\set[1]{\left\{ #1 \right\}}
\newcommand\abs[1]{\left| #1 \right|}
\newcommand\p[1]{\left( #1 \right)}
\newcommand\Z{\mathbb{Z}}
\newcommand\Q{\mathbb{Q}}
\newcommand\F{\mathbb{F}}
\newcommand\fp{\mathfrak{p}}
\newcommand\GSp{\operatorname{GSp}}
\newcommand\Sp{\operatorname{Sp}}
\newcommand\PGSp{\operatorname{PGSp}}
\newcommand\PSp{\operatorname{PSp}}
\newcommand\GL{\operatorname{GL}}

\newcommand\mult{\operatorname{mult}}
\newcommand\tr{\operatorname{tr}}
\newcommand\Inn{\operatorname{Inn}}
\newcommand\Rad{\operatorname{Rad}}
\newcommand\Gal{\operatorname{Gal}}
\newcommand\Aut{\operatorname{Aut}}
\newcommand\Frob{\operatorname{Frob}}
\newcommand\End{\operatorname{End}}
\newcommand\rad{\operatorname{rad}}
\newcommand\pr{\operatorname{pr}}
\newcommand\im{\operatorname{im}}
\newcommand\disc[1]{d_{#1}}
\newcommand\Jac{\operatorname{Jac}}

\theoremstyle{plain}
\newtheorem{theorem}{Theorem}[section]
\newtheorem{lemma}[theorem]{Lemma}
\newtheorem{proposition}[theorem]{Proposition}
\newtheorem{corollary}[theorem]{Corollary}
\newtheorem{algorithm}[theorem]{Algorithm}

\theoremstyle{definition}

\newtheorem{example}[theorem]{Example}

\theoremstyle{remark}
\newtheorem{remark}[theorem]{Remark}

\numberwithin{equation}{section}

\title[An effective open image theorem for products of PPAVs]{An effective open image theorem for products of principally polarized abelian varieties}
\author[Mayle]{Jacob Mayle}
\address{Jacob Mayle, Department of Mathematics, Wake Forest University, Winston-Salem, NC, USA}
\email{maylej@wfu.edu}
\author[Wang]{Tian Wang}
\address{Tian Wang, Department of Mathematics \& Statistics,
Concordia University,
Montreal, Quebec, Canada}
\email{tian.wang@concordia.ca}
\date{\today}

\begin{document}

\begin{abstract} 
Let $A = \prod_{1\leq i\leq n} A_i$ be the product of principally polarized abelian varieties $A_1, \ldots, A_n$ of dimensions $g_1, \ldots, g_n$, respectively, each defined over a number field $K$, and pairwise nonisogenous over $\overline{K}$.  We make effective an open image theorem for $A$ due to Hindry and Ratazzi. More specifically,  we give an explicit bound of the constant $c(A)$ under GRH, in terms of standard invariants of $K$ and each $A_i$, where $c(A)$ is defined to be the smallest positive integer such that for any prime $\ell>c(A)$, the image of the $\ell$-adic Galois representation of $A$ is ``as large as possible'' in a suitable sense.
\end{abstract}

\maketitle

\section{Introduction} \label{S:intro}

Let $K$ be a number field. For $n\geq 1$, let $A$ be the product of principally polarized abelian varieties $A_1/K, \ldots, A_n/K$ of dimensions $g_1, \ldots, g_n$, respectively. For a prime $\ell$, let 
$T_\ell(A) \coloneqq \varprojlim A[\ell^k]$
be the $\ell$-adic Tate module of $A$, where $A[\ell^k]$ denotes the $\ell^k$-torsion subgroup of $A(\overline{K})$.  Similarly, let $T_\ell(A_i)$ be the $\ell$-adic Tate module of $A_i$. We have that
\[
T_\ell(A) \cong T_\ell(A_1) \oplus \cdots \oplus T_\ell(A_n) \cong \bigoplus_{i=1}^n \Z_\ell^{2g_i}
 \]
 as $\Z_\ell$-modules. We choose a $\mathbb{Z}_\ell$-basis for each $T_\ell(A_i)$ and concatenate these together to form a $\Z_\ell$-basis for  $T_\ell(A)$. With respect to this basis, the natural action of $\Gal(\overline{K}/K)$ on $T_\ell(A)$ gives rise  to the $\ell$-adic Galois representation of $A$,
\[ 
\rho_{A,\ell}=\rho_{A_1\times\ldots\times A_n,\ell} \colon \Gal(\overline{K}/K) \longrightarrow \prod_{i=1}^n \GSp_{2g_i}(\Z_\ell),\]
where for any $g\geq 1$, $\GSp_{2g} \subseteq \GL_{2g}$ denotes the general symplectic group.

We begin by discussing the case of a single abelian variety (i.e., $n=1$). Here, we write $g$ to denote the dimension $g_1$. To start,  recall that Serre's open image theorem for abelian varieties \cite{MR3185222Letter,MR1484415,MR387283} gives that if the geometric endomorphism ring of $A$ is trivial (i.e., $\End(A_{\overline{K}}) = \Z$) and $g$ is $2,6,$ or odd, then $\rho_{A,\ell}$ is surjective for each sufficiently large prime $\ell$. This result has been generalized to a density 1 set of positive integers $g$ due to the work of Pink \cite[Theorem 5.14]{MR1603865} and Wintenberger \cite[Th\'eor\`eme 2]{MR1944805}.  In order to quantify these open image  results, 
we study the constant $c(A)$  defined as the least positive integer such that
\[
\ell > c(A) \quad \implies \quad \rho_{A,\ell} \text{ is surjective},
\]
provided such an integer exists and $c(A) = \infty$ otherwise.
In the case of elliptic curves (i.e., $g= 1$), a well-known conjecture \cite{MR3482279,zywina2015possible} (originally a question of Serre)  asserts that $c(A)\leq37$ for each elliptic curve $A/\Q$ without complex multiplication. While this conjecture remains open, significant partial progress has been made by studying rational points on modular curves \cite{MR482230,MR3137477, MR644559, MR3961086, MR2753610, MR4261100, MR4589060, furio2023serresuniformityquestionproper}. Further, there is a rich history of bounds on $c(A)$ in terms of invariants of the $A$ \cite{MR3161774,MR2118760,MR1360773,MR1209248,MR3498908,MR1865384,MR644559}. Following a comment  of Serre \cite[Note 632.6]{MR3185222}, the authors \cite{mayle2021effective} recently obtained the explicit bound
\begin{equation}\label{ec-bound} 
c(A) \leq 964 \log \rad(2N_{A}) + 5760
\end{equation}
under the Generalized Riemann Hypothesis for Dedekind zeta functions (GRH), where $N_A$ is the conductor of $A$ and $\rad n \coloneqq  \prod_{p \mid n} p $ is the radical of an integer $n$.
Expanding beyond the realm of elliptic curves (with the same assumption that $n=1$), an increasing interest has arisen in exploring the constant  $c(A)$ for abelian varieties of higher dimensions (i.e., $g\geq 2$)   \cite{MR3510393, NonSurjAlg, MR1969642, Lom16, ZywinaEffectiveAV}.

From now on, we assume $n \geq 2$ and the dimensions $g_1,\ldots,g_n$ are arbitrary. Note that in this case, the geometric endomorphism ring of $A$ is nontrivial, which imposes a restriction on the image of $\rho_{A,\ell}$. Specifically, the Weil pairing on each $A_i$ implies that the image of $\rho_{A,\ell}$ must be contained in the group 
\[ \Delta_{g_1,\ldots, g_n}(\Z_\ell) \coloneqq  \set{(\gamma_1, \ldots, \gamma_n) \in \prod_{i=1}^n \GSp_{2g_i}(\Z_\ell)  : \mult(\gamma_1) = \cdots = \mult(\gamma_n)}\!, \]
where $\mult\colon\GSp_{2g_i}(\Z_\ell) \twoheadrightarrow \Z_\ell^\times$ denotes the multiplier map. In the case that $g\coloneqq g_1=\ldots =g_n$, we write  $\Delta_{2g, n}(\Z_\ell)$ as shorthand for $\Delta_{2g, n}(\Z_\ell) \coloneqq \Delta_{g, \ldots, g}(\Z_\ell)$ (see Section \ref{sub-2.2}).
In view of the natural restriction on the image of $\rho_{A,\ell}$ just described, we shall always consider $\rho_{A,\ell}$ as a map
\[
\rho_{A,\ell} \colon \Gal(\overline{K}/K) \longrightarrow \Delta_{g_1, \ldots, g_n}(\Z_\ell).
\]
Consequently, when we say that $\rho_{A,\ell}$ is surjective, we mean that its image is $\Delta_{g_1, \ldots, g_n}(\Z_\ell)$. With this in mind, a result of Hindry and Ratazzi \cite[Th\'eor\`eme 1.4]{MR2862374} implies an open image theorem for $A$. They showed (in particular) that if each $A_1, \ldots, A_n$ has a trivial geometric endomorphism ring, are pairwise  non $\overline{K}$-isogenous, and   $\rho_{A_i, \ell}$ is surjective for all sufficiently large primes $\ell$ and each $1 \leq i \leq n$, then $\rho_{A,\ell}$ is surjective for all 
 sufficiently large prime numbers $\ell$. Their work generalizes prior results of Serre \cite[Th\'eor\`eme 6, p. 324]{MR387283} and Ribet \cite[Theorem 3.5]{MR419358} that give open image theorems for  products of non $\overline{K}$-isogenous elliptic curves without complex multiplication. 
 
 This paper aims to quantify the aforementioned result of Hindry and Ratazzi. Namely, we define the constant $c(A_1 \times \cdots \times A_n)$ to be the least positive integer such that
\begin{equation}\label{def-c(.)}
\ell > c(A_1\times \cdots \times A_n) 
\quad \implies \quad
\rho_{A,\ell} \text{ is surjective},\end{equation}
provided such an integer exists and $ c(A_1\times \cdots \times A_n) = \infty$ otherwise. In the case of a product of elliptic curves, this constant has already been considered. Indeed, suppose that $A_1, \ldots, A_n$ are pairwise non $\overline{K}$-isogenous elliptic curves without complex multiplication. Masser and W\"{u}stholz \cite[Proposition 1, p. 251] {MR1209248} gave a bound on $c(A_1\times \cdots \times A_n)$ in terms of the degree of $K$ and the Faltings height $h(A_i)$ of each $A_i$. Building on their work and the techniques in \cite{MR3437765, MR1209248}, Lombardo \cite[Lemma 7.1]{MR3515826} obtained an effective bound on $c(A_1 \times \cdots \times A_n)$ in terms of the same invariants. 

We proceed to present our contribution, which is an explicit bound on the constant $c(A_1 \times \cdots \times A_n)$. We distinguish between the case where each $A_i$ is of a common dimension $g$ and the case of differing dimensions. 
The former case is  more intricate and is the subject of our main theorem.
The latter case emerges later in the paper (Theorem \ref{main-thrm-2}) as a consequence of a group-theoretic lemma. To succinctly express the bound appearing in our main theorem, we 
set some notation. Fix integers $n \geq 2$ and $g \geq 1$. Let $\ell_g$ be the least prime not dividing $2g$. For each $1\leq i\neq j \leq n$, we define
\begin{align}
   c(K, g) & \coloneqq 2(\ell_g^{8g^2}-1)^2\left(\log d_K + [K:\Q] \log \left(2(\ell_g^{8g^2}-1)^2\right) \right),  \label{c-K-g} \\
   c(K, A_i, A_j) & \coloneqq 2(\ell_g^{8g^2}-1)^2 [K:\Q]\log \left(\rad (2\ell_g N_{A_i}N_{A_j}d_K)\right). \label{c-K-A}
\end{align}

A final missing piece that appears in the statement of our main theorem is the notion of an adelic Galois representation. We defer the definition to Section \ref{S:GalRepProd} in order to state the main theorem now.

\begin{theorem} \label{main-thrm-1} Assume GRH. Let $K$ be a number field. Let $A_1/K, \ldots, A_n/K$ be principally polarized abelian varieties of a common dimension $g$. Let $\ell_g$ be the least prime not dividing $2g$. Assume that the adelic Galois image $G_{A_i}$ of $A_i$ is open in $\GSp_{2g}(\widehat{\Z})$ for each $1 \leq i \leq n$. Then the following statements are equivalent:
\begin{enumerate}
    \item For all $1\leq i \neq j \leq n$,  $A_i$ and $A_j$ are non  $\overline{K}$-isogenous.  \label{main-thm-isogenous}
    \item The adelic Galois image $G_{A_1 \times \cdots \times  A_n}$ of  $A$  is an open subgroup of $\Delta_{2g, n}(\widehat{\Z})$.\label{main-thm-adelic-image}
    \item The constant
    $c(A_1\times \cdots \times A_n)$ exists and is bounded above by
    \[
    \max_{1\leq i\neq j \leq n}\{4g \left(\tilde{a} \left(c(K, g)+  c(K, A_i, A_j)\right) + 2\tilde{b} (\ell_g^{8g^2}-1)^2   + \tilde{c}\right),  c(A_i)\}
    \]
    where $\tilde{a}$, $\tilde{b}$, and $\tilde{c}$ are the explicit constants in Corollary \ref{ECDT-cor}.
   \label{main-thm-ell-adic-image}    
\end{enumerate}
\end{theorem}

Specializing to the case of a product of elliptic curves over $\Q$, we obtain the following corollary.

\begin{corollary}\label{main-cor}
Assume GRH. Let $A_1/\Q, \ldots, A_n/\Q$  be elliptic curves without complex multiplication. Then the following  statements are equivalent:
\begin{enumerate}
    \item \label{isogeny-cor} For all $1\leq i \neq j \leq n$,  $A_i$ and $A_j$ are non  $\overline{\Q}$-isogenous.
    \item \label{adelic-image-cor}  The adelic Galois image $G_{A_1\times\cdots \times A_n}$ of $A$  is an open subgroup of  $\Delta_{2,n}(\widehat{\Z})$.
    \item \label{ell-adic-cor}  The constant  $c(A_1\times \cdots \times A_n)$ exists and satisfies the bound 
    \[
    c(A_1\times\cdots \times A_n) \leq \max_{1 \leq i \neq j \leq n} \left\{1377075200 \log \rad(6 N_{A_i} N_{A_j}) + 26020715799\right\}.
    \] 
\end{enumerate}
\end{corollary}

We now make some remarks on Theorem \ref{main-thrm-1} and Corollary \ref{main-cor}.

\begin{remark}\label{remark-1} 
The only parts of Theorem \ref{main-thrm-1}  and Corollary \ref{main-cor} that assume GRH are the bounds appearing in the third part of each. It is worth noting that weaker unconditional bounds can be given by using an unconditional effective Chebotarev density theorem (see  \cite[Theorem 1]{MR4472459}). 
\end{remark}

\begin{remark}
Under GRH, we can  invoke \cite[Theorem 1.4]{ZywinaEffectiveAV} to obtain an effective bound of $c(A_i)$ in part (3) of Theorem \ref{main-thrm-1} and  in Theorem \ref{main-thrm-2}  that only depends on $g$, $h(A)$, $[K:\Q]$ and $d_K$.
\end{remark}

\begin{remark} The assumption that each adelic Galois image $G_{A_i}$ is open in $\GSp_{2g}(\widehat{\Z})$ is crucial to Theorem \ref{main-thrm-1}. Indeed, it follows from a result of Lombardo \cite[Theorem 1.1]{LombardoIsoKummerian} that without this assumption,  for all $g \geq 4$, there are infinitely many pairs of abelian varieties $A_1$ and $A_2$ that satisfy \eqref{main-thm-isogenous} of Theorem \ref{main-thrm-1}, yet fail \eqref{main-thm-adelic-image} and \eqref{main-thm-ell-adic-image}. On the other hand, the assumption that each $G_{A_i}$ is open in $\GL_2(\widehat{\Z})$ does not appear in Corollary \ref{main-cor}, because Serre's open image theorem already guarantees that this is that case since each $A_i$ is without complex multiplication.
\end{remark}

\begin{remark}
We now compare the bound for $c(A_1\times \cdots \times A_n)$ in  Corollary \ref{main-cor}  with an existing result that is given in  \cite[Proposition 1]{MR1209248} (cf. \cite[Proposition 2.5]{MR3515826}). Indeed,  \cite[Proposition 1]{MR1209248} gives $c(A_1\times \cdots \times A_n)\leq c \max_{1\leq i\leq n} \{1, h(A_i)\}^{\gamma}$, where $c$ and $\gamma$ are completely effective, though $\gamma$ is large and dependent on $n$. 
Although our bound is in terms of the conductors of $A_1,\ldots,A_n$, a comparison between the two results can be made under an appropriate conjecture. For each $1\leq i\leq n$, by the generalized Szpiro conjecture and the relation between the minimal discriminant  and Faltings height of $A_i$ \cite[Chapter VI]{CornellSilverman},  we expect the following inequality
\[
\alpha \log N_{A_i} \leq  |h(A_i)| \leq \beta \log N_{A_i}
\]
holds for some  absolute positive constants $\alpha, \beta\in \mathbb{R}$. Therefore, we expect that our bound given in Corollary \ref{main-cor} offers a power saving in  the logarithm of the conductors.
\end{remark}

To conclude the introduction, we offer a sketch of the proof for Theorem \ref{main-thrm-1}. We will show in Section \ref{reduce-to-two} that the proof can be reduced to considering a product of two principally polarized abelian varieties.  Furthermore, utilizing Proposition \ref{lift-SpSp}, for $\ell \geq 5$, checking the surjectivity of the $\ell$-adic Galois representation $\rho_{A, \ell}$ can be reduced to checking the surjectivity of the mod $\ell$ Galois representation $\overline{\rho}_{A, \ell}$. Then in Section \ref{S3}, we provide a group-theoretic criterion, in terms of traces and characteristic polynomials, to determine when a subgroup of $\Delta_{2g, 2}(\F_{\ell})$ is equal to $\Delta_{2g, 2}(\F_{\ell})$ (Corollary \ref{corr-class}); additionally in Section \ref{S4}, we offer an effective version of Faltings's isogeny theorem (Theorem \ref{effective-Faltings}) for abelian varieties over the number field $K$ with trivial geometric endomorphism rings to distinguish their isogeny classes over $\overline{K}$.  Finally, taking inspiration from \cite[Section 8]{MR644559}, we then apply the preceding results and Weil's bound to give an effective bound on $c(A_1 \times \cdots \times A_n)$ in Section \ref{product-two}. 
Building on the aforementioned group-theoretic criterion, Section \ref{S6} contains three numerical examples and an algorithm. The algorithm takes in two hyperelliptic curves $C_1$ and $C_2$, along with a positive integer $B$,  and produces a nonnegative integer $\Lambda$ which, if nonzero, serves to bound the largest nonsurjective prime associated with the product of Jacobians $\Jac(C_1) \times \Jac(C_2)$. 
We implemented the algorithm in Sage. The code is available at the public repository:

\centerline{\url{https://github.com/maylejacobj/ProductPPAVs}.}

We conclude the introduction by stating Theorem \ref{effective-Faltings}, which is an effective version of Faltings's isogeny theorem for abelian varieties that  may carry independent interest. Let $A_1$ and $A_2$ be principally polarized abelian varieties over $K$ that are not $K$-isogenous. Achter \cite[Lemma 1.2]{MR2181871} gave an explicit upper bound on the smallest norm of an ideal $\mathfrak{p}$ of $K$ such that the reductions of $A_1$ and $A_2$ are non-isogenous modulo $\mathfrak{p}$. Achter's result extends an earlier result of Serre 
\cite[Section 8.3, pp.\ 191--196]{MR644559}
 that considered elliptic curves over $\mathbb{Q}$. We now state our theorem, which provides an explicit bound in the case where two ``generic'' abelian varieties are not {geometrically} isogenous. (See also relevant work by Bucur, Fit\'e, and Kedlaya   \cite[Corollary 1.3]{MR4735820}.)

\begin{theorem}\label{effective-Faltings}
Assume GRH. Let $A_1/K$ and $A_2/K$ be principally polarized abelian varieties of a common dimension $g$.  Assume that the adelic Galois image $G_{A_i}$  is open in $\GSp_{2g}(\widehat{\Z})$ for each $i\in \{1,2 \}$.  If $A_1$ and $A_2$ are not $\overline{K}$-isogenous, then there exists a prime $\fp$ of $K$ that is of good reduction for $A_1$ and $A_2$ such that $a_\fp(A_1)\neq \pm a_\fp(A_2)$, satisfying the bound
\[N(\fp)\leq  \left(\tilde{a} \left(c(K, g)+  c(K, A_1, A_2)\right) + 2\tilde{b} (\ell_g^{8g^2}-1)^2  + \tilde{c}\right)^2,\]
where  $\tilde{a}$, $\tilde{b}$, and $\tilde{c}$ are the explicit constants in Corollary \ref{ECDT-cor}.
\end{theorem}

\begin{remark}
Similar to Remark \ref{remark-1}, without assuming GRH, we can get an explicit bound of $N(\mathfrak{p})$ that grows exponentially in terms of  $N_{A_1}$, $N_{A_2}$, and $d_K$.
\end{remark}
\subsection*{Acknowledgments} We wish to thank Davide Lombardo and Ramin Takloo-Bighash for their helpful comments.  The second author is partially founded by Dean's Scholar Fellowship at the University of Illinois at Chicago and  also wishes to thank the Max-Planck-Institut f\"{u}r Mathematik in Bonn for its support and inspiring atmosphere. Lastly, we are  grateful for the many helpful comments and suggestions the referee provided.

\section{Preliminaries}\label{S2}

\subsection{Some results from group theory}

We begin by recalling the definition of the fiber product of groups and Goursat's lemma. Let $G_1$ and $G_2$ be groups, and for $i \in \set{1,2}$, let $\pr_i$ denote the natural projection map given by
\[\pr_i: G_1 \times G_2 \twoheadrightarrow G_i.\]
Consider a group $Q$ and   surjective group homomorphisms $\phi_1: G_1 \twoheadrightarrow Q$ and $\phi_2: G_2 \twoheadrightarrow Q$. The \emph{fiber product} of $G_1$ and $G_2$ over the pair $(\phi_1,\phi_2)$ is the subgroup
\begin{equation*} \label{fiber-def}
   G_1 \times_{(\phi_1,\phi_2)} G_2 \coloneqq  \set{(\gamma_1,\gamma_2) \in G_1 \times G_2 : \phi_1(\gamma_1) = \phi_2(\gamma_2)}
\end{equation*}
of $G_1 \times G_2$. It is clear that $\pr_i(G_1 \times_{(\phi_1,\phi_2)} G_2) = G_i$ for each $i \in \set{1,2}$. Goursat's lemma asserts that all subgroups of $G_1 \times G_2$, for which $\pr_1$ and $\pr_2$ are both surjective, arise as fiber products, as we now recall.

\begin{lemma} \label{Goursat} Let $G \subseteq G_1 \times G_2$ be a subgroup. Then  $\pr_i(G) = G_i$ holds for each $i \in \set{1,2}$ if and only if $G$ is the fiber product $G = G_1 \times_{(\phi_1,\phi_2)} G_2$ for some surjective homomorphisms  $\phi_1$ and $\phi_2$. 
\end{lemma}
\begin{proof} See \cite[p.\ 45]{MR1508819} and, for instance, \cite[p.\ 75]{MR1878556}.
\end{proof}

Next,  we state a proposition on the automorphism group $\Aut(G)$ of a finite group $G$ of a certain form. Recall that a subgroup $N$ of $G$ is called a \emph{characteristic subgroup} if $\sigma(N)=N$ for every $\sigma\in \Aut(G)$. Let $G=N \rtimes H$ be an internal semidirect product of subgroups $N$ and $H$ of $G$. Recall that a map $\beta\colon H\to Z(N)$ is a \emph{crossed homomorphism} if for any $k, k'\in H$, 
\[
\beta(kk')=\beta(k)\cdot(k\beta(k')k^{-1}),
\]
where  $Z(\cdot)$ denotes the  center of a group. We denote by $Z^{1}(H,Z(N))$ the group of all crossed homomorphisms from $H$ to $Z(N)$.

\begin{proposition} \label{aut-semi} Let $G$ be a finite group. Let $H$ be a subgroup of $G$ and $N$ be a characteristic subgroup of $G$. If $G = N \rtimes H$, then there exists a homomorphism $\Aut(G) \to \Aut(N) \times \Aut(H)$ with kernel isomorphic to $Z^{1}(H,Z(N))$.
\end{proposition}
\begin{proof} See \cite[Theorem 1, p.\ 206]{MR2475812}  and the remarks following its proof.
\end{proof}

\subsection{Properties of symplectic groups}\label{sub-2.2} In this subsection, we give a brief introduction to the symplectic group and some of its properties that will be used in the paper. For a more comprehensive introduction, we refer the reader to  \cite{MR502254}.

Let $R$ be a commutative ring with unity and let $R^{\times}$ denote the group of units of $R$. For a fixed  positive integer $g$,  let $M_{2g\times 2g}(R)$ be the $R$-algebra of all $2g\times 2g$ matrices with entries in $R$. 
We define the \emph{general symplectic group} over $R$ of dimension $2g$ as
\[ \GSp_{2g}(R) \coloneqq \set{\gamma \in M_{2g\times 2g}(R) \colon \gamma^t \Omega_{2g} \gamma = \mult(\gamma) \Omega_{2g} \text{ for some } \mult(\gamma) \in R^\times}, \]
where $\Omega_{2g} \coloneqq \begin{psmallmatrix} 0 & I_g\\ -I_g & 0\end{psmallmatrix}$, $\gamma^t$ is the transpose of $\gamma$, and $\mult(\gamma)$ is called the \emph{multiplier} of $\gamma$. The multiplier map $\gamma \mapsto \mult(\gamma)$ is a well-defined surjective group homomorphism
\begin{equation*} \label{mult-surj} 
\mult\colon \GSp_{2g}(R) \twoheadrightarrow R^\times.
\end{equation*}
For each $\gamma \in \GSp_{2g}(R)$, we have that $ \mult(\gamma)^g = \det(\gamma)$.  In particular, we have the containment
\[ \GSp_{2g}(R) \subseteq \GL_{2g}(R) \]
and the equality holds if $g = 1$. The \emph{symplectic group} $\Sp_{2g}(R)$ is the kernel of the multiplier map:
\[ \Sp_{2g}(R) \coloneqq \set{\gamma \in M_{2g\times 2g}(R) \colon \gamma^t \Omega_{2g} \gamma =  \Omega_{2g}}.  \]

Now, we focus our attention on symplectic groups over $R = \F_\ell$. We observe that
\begin{equation} \label{GSp-semi}
     \GSp_{2g}(\F_\ell) \cong \Sp_{2g}(\F_\ell) \rtimes \iota(\F_\ell^\times),
\end{equation}
where $\iota \colon \F_\ell^\times \to \GSp_{2g}(\F_\ell)$ is given by $ a \mapsto \begin{psmallmatrix} a I_g & 0 \\ 0 & I_g \end{psmallmatrix}$, since  $\iota$ splits the exact sequence
\[ 1 \to \Sp_{2g}(\F_\ell) \to \GSp_{2g}(\F_\ell) \to \F_\ell^\times \to 1. \]
The orders of $\Sp_{2g}(\F_\ell)$ and $\GSp_{2g}(\F_\ell)$ \cite[Theorem 3.1.2]{MR502254} are given by
\begin{equation}  \label{E:gsp-order} \abs{\Sp_{2g}(\F_\ell)} = \ell^{g^2} \prod_{i = 1}^{g} (\ell^{2i} - 1) \quad \text{and} \quad \abs{\GSp_{2g}(\F_\ell)} = (\ell-1)\ell^{g^2} \prod_{i = 1}^{g} (\ell^{2i} - 1).
\end{equation}
Let $R_{2g}(\F_\ell)$ denote the group of scalar matrices
\[R_{2g}(\F_\ell) \coloneqq  \set{\lambda I_{2g} : \lambda \in \F_\ell^\times} \trianglelefteq \GSp_{2g}(\F_\ell). \]
Since $\mult(\lambda I_{2g} ) = \lambda^2$, we have that
\[ \Sp_{2g}(\F_\ell) \cap R_{2g}(\F_\ell) = \set{\pm I_{2g}}. \]
The \emph{projective general symplectic group} and \emph{projective symplectic group} are, respectively,
\[
\PGSp_{2g}(\F_\ell) \coloneqq  \GSp_{2g}(\F_\ell)/R_{2g}(\F_\ell) 
\quad \text{and} \quad
\PSp_{2g}(\F_\ell) \coloneqq  \Sp_{2g}(\F_\ell)/\set{\pm I_{2g}}.
\]

The next lemma records several basic properties of the groups just defined.

\begin{lemma} \label{L:GSpProperties} If $\ell \geq 5$, then each of the following statements hold.
\begin{enumerate}
    \item \label{GSppm} We have that $\GSp_{2g}(\F_\ell) / \set{\pm I_{2g}} \cong  \PSp_{2g}(\F_\ell) \rtimes \bar{\iota}(\F_{\ell}^{\times})$, where $\bar{\iota}$ is the composition of $\iota$ with the natural reduction map $\GSp_{2g}(\F_\ell) \to \GSp_{2g}(\F_\ell) / \set{\pm I_{2g}}$.
    \item \label{L:centers} We have that $Z(\GSp_{2g}(\F_\ell))= R_{2g}(\F_\ell)$. Hence $Z(\PGSp_{2g}(\F_\ell))$ and $Z(\PSp_{2g}(\F_\ell))$ are trivial.
    \item \label{L:GSpCharSub} The group $\Sp_{2g}(\F_\ell)$ is a characteristic subgroup of $\GSp_{2g}(\F_\ell)$.
    \item  The normal subgroups of $\Sp_{2g}(\F_\ell)$ are $\set{I_{2g}}$, $\set{\pm I_{2g}}$, and $\Sp_{2g}(\F_\ell)$. \label{normal-Sp}
    \item  Every automorphism of $\Sp_{2g}(\F_\ell)$ is given by conjugation by a matrix in $\GSp_{2g}(\F_\ell)$. It follows that $\Aut(\Sp_{2g}(\F_\ell)) \cong \PGSp_{2g}(\F_\ell)$. Moreover, every automorphism of $\PSp_{2g}(\F_\ell)$ is induced by an automorphism of $\Sp_{2g}(\F_\ell)$.
    \label{Aut-Sp}
\end{enumerate}
\end{lemma}
\begin{proof}
\begin{enumerate}
    \item Follows similarly to \eqref{GSp-semi}. 
    \item See \cite[3.2.1 p.\ 37]{MR502254} and \cite[4.2.5 p.\  52]{MR502254}.
    \item See \cite[3.3.6, p.\ 41]{MR502254}.
    \item See \cite[1.7.5, p.\  20]{MR502254} and  \cite[3.4.1, p.\  41]{MR502254}.
     \item Follows from \cite[Ch.\ VI, Thm.\ 9]{MR606555} and part \eqref{L:centers}. \qedhere
\end{enumerate}
\end{proof}

It is sometimes useful to consider the Lie algebra
\begin{align*} \label{block-symp}
    \mathfrak{sp}_{2g}(R) \coloneqq& \{\gamma \in M_{2g\times 2g}(R) : \gamma^t \Omega_{2g}=-\Omega_{2g} \gamma \} \\
    =& \set{ \begin{psmallmatrix} A & B\\ C & -A^t \end{psmallmatrix}: A, B, C \in M_{g\times g}(R)  \text{ and $B,C$ are symmetric matrices} }.
\end{align*}
For instance, if $k$ is a positive integer and $\ell$ is a prime, then 
\[
\ker (\Sp_{2g}(\Z/\ell^{k+1}\Z)\twoheadrightarrow \Sp_{2g}(\Z/\ell^k\Z)) = I_{2g} + \ell^k \mathfrak{sp}_{2g}(\F_{\ell}),
\]
where the map is given by reduction modulo $\ell^k$ (see \cite[Sec.\ 2.1]{MR4195609}).

Let $n$ and $g_1, \ldots, g_n$ be positive integers and again let $R$ be a commutative ring with unity. As we mentioned in the introduction, we are interested in the $n$-fold fiber product $\Delta_{g_1,\ldots, g_n}(R)$ given by
\begin{equation*}
    \Delta_{g_1,\ldots, g_n}(R) \coloneqq  \set{(\gamma_1, \ldots, \gamma_n) \in \prod_{i=1}^n \GSp_{2g_i}(R)  : \mult(\gamma_1) = \cdots = \mult(\gamma_n)}\!.
\end{equation*}
With some abuse of notation, we often consider the homomorphism $\mult\colon \Delta_{g_1,\ldots, g_n}(R) \twoheadrightarrow R^\times$ given by $(\gamma_1,\ldots,\gamma_n) \mapsto \mult(\gamma_1)$. The kernel of this map is the group
\[
\delta_{g_1,\ldots, g_n}(R) \coloneqq \Sp_{2g_1}(R)\times \cdots \times \Sp_{2g_n}(R) \subseteq \Delta_{g_1, \ldots, g_n}(R).
\]
For $g \coloneqq g_1=\ldots =g_n$, we write 
\begin{equation*} \label{Delta-def}
\Delta_{2g,n}(R)  \coloneqq  \Delta_{g, \ldots, g}(R) \quad \text{ and } \quad \delta_{2g,n}(R)  \coloneqq  \delta_{g, \ldots, g}(R).
\end{equation*}
Further, we  write  $\Delta_{2g}(R)$ and $\delta_{2g}(R)$  as shorthand for $\Delta_{2g,2}(R)$ and  $\delta_{2g,2}(R)$, respectively, i.e., 
\begin{equation*} \label{E:Delta}
\Delta_{2g}(R) = \GSp_{2g}(R) \times_{(\mult,\mult)} \GSp_{2g}(R) \quad \text{ and } \quad \delta_{2g}(R) = \Sp_{2g}(R) \times \Sp_{2g}(R).
\end{equation*}
Note that if $R$ is finite, then by the surjectivity of the multiplier map, the order of $\Delta_{2g}(R)$ is given by
\begin{equation} \label{order-Delta}
\abs{\Delta_{2g}(R)} = \frac{\abs{\GSp_{2g}(R)}^2}{\abs{R^\times}}.
\end{equation}

Next, we state two group-theoretic results which are generalizations of some well-known results. The first is a straightforward generalization of \cite[Lemme 2.20, p. 40]{MR2862374}  (which itself generalizes \cite[Theorem 1]{MR3667841} and \cite[Lemma 3, p. IV-23]{MR1484415}). We briefly give some notation. Let $\widehat{\Z}$ denote the ring of profinite integers. For a closed subgroup $G \subseteq \Delta_{g_1,\ldots,g_n}(\widehat{\Z})$ and a prime $\ell$, we write $G_\ell$ and $G(\ell)$ for the images of $G$ in $\Delta_{g_1,\ldots,g_n}(\Z_\ell)$ and $\Delta_{g_1,\ldots,g_n} (\F_\ell)$, respectively.
\begin{proposition} \label{lift-SpSp} Assume $\ell \geq 5$. Let $G \subseteq \Delta_{g_1,g_2}(\Z_\ell)$ be a closed subgroup. If $\delta_{g_1,g_2}(\F_\ell) \subseteq G(\ell)$, then $\delta_{g_1,g_2}(\Z_\ell) \subseteq G$.
\end{proposition}

The second result is a straightforward generalization of \cite[Lemma 2.8]{MR3981312} (which itself generalizes  \cite[Main Lemma, p.\ IV-19]{MR1484415}).

\begin{proposition}\label{group-open} Let $G \subseteq \Delta_{g_1, \ldots, g_n}(\widehat{\Z})$ be a closed subgroup. Then $G$ is open in $\Delta_{g_1, \ldots, g_n}(\widehat{\Z})$ if and only if
\begin{enumerate}
    \item for each sufficiently large prime $\ell$,  $G(\ell) = \Delta_{g_1, \ldots, g_n}(\F_{\ell})$, \label{red-ell-full}
    \item for each prime $\ell$, $G_{\ell}$ is open in $\Delta_{g_1, \ldots, g_n}(\Z_{\ell})$, and \label{all-ell-open}
   \item the image of $G$ under $\mult \colon \Delta_{g_1, \ldots, g_n}(\widehat{\Z}) \to \widehat{\Z}^{\times}$ is open in $\widehat{\Z}^\times$.\label{mult-open}
\end{enumerate} 
\end{proposition}

\subsection{Galois representations of products of abelian varieties} \label{S:GalRepProd} 
Let $K$ be a number field. For $n\geq 1$, let $A$ be the product of principally polarized abelian varieties $A_1/K, \ldots, A_n/K$ of dimensions $g_1, \ldots, g_n$, respectively. Let $T(A) \coloneqq \varprojlim A[n]$ be the {adelic Tate module} of $A$,
where $A[n]$ denotes the $n$-torsion subgroup of $A(\overline{K})$ and the inverse limit is taken over the integers ordered by divisibility.  Similarly, let $T(A_i)$ be the adelic Tate module of $A_i$. We have that 
\[
T(A) \cong T(A_1) \oplus \cdots \oplus T(A_n) \cong \bigoplus_{i=1}^n \widehat{\Z}^{2g_i}.
 \]
The Galois group $\Gal(\overline{K}/K)$ acts on $T(A)$, giving rise to the {adelic Galois representation} of $A$,
\[ \rho_{A} \colon \Gal(\overline{K}/K) \to \Aut(T(A)). \]

Upon fixing an appropriate $\widehat{\Z}$-basis for $T(A)$ and noting that $A_i$ is principally polarized for each $1\leq i\leq n$, we may consider $\rho_A$ as a map from $\Gal(\overline{K}/K)$ to $ \prod_{i=1}^n \GSp_{2g_i}(\widehat{\Z})$. 
As before, the existence of the Weil pairing of $A$ implies that the composition $\mult \circ \rho_{A_i}$ is the cyclotomic character $\Gal(\overline{K}/K) \to \widehat{\Z}^\times$ for each $i$, which further restricts the image of $\rho_A$. 
As a consequence, 
\[ \mult \circ \rho_{A_1} = \cdots = \mult \circ \rho_{A_n}, \]
which gives the inclusion of the image  of $\rho_A$  in $\Delta_{g_1, \ldots, g_n}(\widehat{\Z})$.
Thus,  in this paper we always take $\Delta_{g_1, \ldots, g_n}(\widehat{\Z})$ to be the codomain of $\rho_A$. This is to say that we always consider $\rho_A$ as a map
\[ \rho_{A} \colon \Gal(\overline{K}/K) \to \Delta_{g_1, \ldots, g_n}(\widehat{\Z}). \]

Similarly, for each prime $\ell$, we consider the $\ell$-adic and mod $\ell$ Galois representations of $A$
\begin{align*}
    \rho_{A,\ell}&\colon \Gal(\overline{K}/K) \to \Delta_{g_1, \ldots, g_n}(\Z_\ell), \\
    \bar{\rho}_{A,\ell}&\colon \Gal(\overline{K}/K) \to \Delta_{g_1, \ldots, g_n}(\F_\ell),
\end{align*}
 which encode the action of $\Gal(\overline{K}/K)$ on the $\ell$-adic Tate module $T_\ell(A)$ and $\ell$-torsion $A[\ell]$, respectively. The maps $\rho_{A,\ell}$ and $\bar{\rho}_{A,\ell}$ can be viewed as the composition of the adelic Galois representation with the natural projection map \[
\Delta_{g_1, \ldots, g_n}(\widehat{\Z}) \twoheadrightarrow \Delta_{g_1, \ldots, g_n}(\Z_\ell) \; \text{ and } \; \Delta_{g_1, \ldots, g_n}(\widehat{\Z}) \twoheadrightarrow \Delta_{g_1, \ldots, g_n}(\F_\ell),
\]
respectively. 

We use the notation $G_A$, $G_{A,\ell}$, and $G_A(\ell)$ to denote the images of $\rho_A$, $\rho_{A,\ell}$, and $\bar{\rho}_{A,\ell}$, respectively. In this notation, $\pr_i(G_A) = G_{A_i}$, and similarly we have  $\pr_i(G_{A,\ell}) = G_{A_i,\ell}$ and $\pr_i(G_A(\ell)) = G_{A_i}(\ell)$.

For now, we specialize to $n = 1$, i.e., suppose that $A$ is a principally polarized abelian variety defined over $K$ of dimension $g$. We recall some arithmetic invariants associated with $A$.  Let $N_A$ be the norm of the conductor ideal of $A$. For any prime $\fp\nmid N_A$, we fix an absolute Frobenius automorphism $\Frob_\fp\in \Gal(\overline{K}/K)$. If $\ell$ is any prime number such that $\fp \nmid \ell$, then the matrix $\rho_{A,\ell}(\Frob_\fp) \in \GSp_{2g}({\Z_\ell})$ has an integral characteristic polynomial
\[
P_{A, \fp}(X) =X^{2g}-a_\fp(A)X^{2g-1}+\ldots + N(\fp)^g \in \Z[X],
\]
which is independent of $\ell$. The constant $a_\fp(A)$ is the \emph{Frobenius trace} of $A$ at $\fp$ and $N(\fp)$  is the norm of the ideal $\fp$. Recall that $a_\fp(A)$ is an integer that satisfies Weil's bound (see, for instance, \cite[Theorem 19.1, p. 143]{CornellSilverman}), 
\begin{equation}\label{weil-bound}
 |a_\fp(A)|\leq 2g\sqrt{N(\fp)}.
 \end{equation}

\subsection{Effective Chebotarev density theorem} The proof of our main result relies on an explicit conditional Chebotarev density theorem due to Bach and Sorenson, which we now recall. Let $K/\Q$ be a number field that is Galois  over $\Q$. Write $d_K$  to denote the absolute value of the discriminant of $K/\Q$. We denote by $\Sigma_K$ the set of all finite places of $K$. For each $\fp\in \Sigma_K$, we write $N(\fp)$ for the norm of the prime ideal $\fp$. 
Now let $L/K$ be a Galois  extension of number fields. We  denote
 by  $[L:K]$ the degree of $L/K$
 and
 by $\disc{L/K}$ the norm of the relevant discriminant ideal of $L/K$. We also denote 
$$
 P(L/K) \coloneqq \{p \ \text{rational prime}: p \text{ divides } \disc{L/K}\}.$$
 
For a prime ideal $\fp$ of $K$ that is unramified in $L/K$, let $\p{\tfrac{\fp}{L/K}}$ denote the Artin symbol of $L/K$ at $\fp$. We may view $\p{\tfrac{\fp}{L/K}}$ as the restriction of an absolute Frobenius element $\Frob_\fp$ to $\Gal(L/K)$. Let $C \subseteq \Gal(L/K)$ be a nonempty subset that is closed under conjugation. We consider the relation
\begin{equation} \label{artin-C}
    \p{\frac{\fp}{L/K}} \subseteq C.
\end{equation}

The theorem below is an effective version of the Chebotarev density theorem that, conditional on the Generalized Riemann Hypothesis for Dedekind zeta functions (GRH), bounds the least norm among unramified primes $\fp$ for which \eqref{artin-C} holds.

\begin{theorem}[Bach-Sorenson] \label{effective-cdt}
Assume GRH. Let notation be as above. Then there exists a prime $\mathfrak{p}$ of $K$ that is unramified in $L/K$ for which (\ref{artin-C}) holds and that satisfies the inequality
\[ N(\mathfrak{p}) \leq  (a \log d_L + b [L:K] + c)^2, \]
where $a,b,$ and $c$ are absolute constants that can be taken to be $4, 2.5,$ and $5$, respectively, or can be taken to be the improved values given in \cite[Table 1]{MR1355006} associated with $L$.
\end{theorem} 
\begin{proof} See \cite[Theorem 5.1]{MR1355006}. \end{proof}

In fact, \cite[Theorem 5.1]{MR1355006}  guarantees the existence of a prime $\fp$ of residue degree 1 with the claimed properties. We will not use this information about the residue degree here. Instead, in our application, we will use a corollary of Theorem \ref{effective-cdt} that allows for the avoidance of a prescribed set of primes.

\begin{corollary}\label{ECDT-cor} Assume GRH. Let notation be as above, let $m$ be a positive squarefree integer, and set $\tilde{L} \coloneqq L(\sqrt{m})$. Then there exists a prime $\fp$ of $K$ that is unramified in $L/K$ and does not divide $m$ for which (\ref{artin-C}) holds and that satisfies  the inequality
\begin{equation*} \label{ECDT-cor-ineq}
    N(\fp) \leq  (\tilde{a} \log d_{\tilde{L}} + \tilde{b} [\tilde{L}:K] + \tilde{c})^2,
\end{equation*} 
where  $\tilde{a},\tilde{b},\tilde{c}$ are absolute constants that can be taken to be $4, 2.5,$ and $5$, respectively, or can be taken to be the improved values given in \cite[Table 1]{MR1355006}  associated with $\tilde{L}$.
\end{corollary}
\begin{proof}
Follow from the proof of  \cite[Corollary 6, p. 4]{mayle2021effective}.
\end{proof}
In order to apply Corollary \ref{ECDT-cor}, we will need a bound on $\log d_L$. The next lemma provides such a bound in terms of the degrees $[L:K]$, $[L:\Q]$, and $[K:\Q]$ and the discriminants $d_K$ and $\disc{L/K}$.

\begin{lemma} \label{lem-logdK} If $L/K$ is a finite Galois extension of number fields, then
\begin{align*}
\log d_L  \leq & [L : K]\log d_K +  
([L : \Q] - [K : \Q])
\log \rad(\disc{L/K})+ [L : \Q]  \log [L : K],
\end{align*} 
 where $\rad n \coloneqq \prod_{p \mid n} p$ is the radical of an integer $n$. 
\end{lemma}
\begin{proof} The claim follows from the inequality given in \cite[Proposition 4', p. 329]{MR644559} (and the remark that immediately follows), and the observation that  $\sum_{p \in P(L/K)} \log p= \log \rad(\disc{L/K})$.
\end{proof}

\section{Group theoretic considerations}\label{S3}

\subsection{Automorphisms of the general symplectic group}

Throughout this section, we make the assumption that $\ell \geq 5$ to avoid certain nuances with $\ell = 2,3$. For a fixed integer $g\geq 1$, we write $I$ to denote the identity matrix $I_{2g}$. 

To begin with, we provide an explicit description of the automorphism group of $\GSp_{2g}(\F_\ell)$. This description is essential in characterizing the automorphisms that preserve the multiplier, which will play a key role in our proof of Proposition \ref{prop-class}.  Later in the section, we perform a similar analysis for $\GSp_{2g}(\F_\ell) / \set{\pm I}$. 

\subsubsection{Automorphisms of $\GSp_{2g}(\F_\ell)$} \label{sec-Aut-GSp}

Write  $\Inn(\GSp_{2g}(\F_\ell))$ to denote the inner automorphism group of $\GSp_{2g}(\F_\ell)$. For $\gamma,\gamma' \in \GSp_{2g}(\F_\ell)$, note that by Lemma \ref{L:GSpProperties}(\ref{L:centers}), $\gamma$ and $\gamma'$ give rise to the same inner automorphism of $\GSp_{2g}(\F_\ell)$ if and only if the images of $\gamma$ and $\gamma'$ in $\PGSp_{2g}(\F_\ell)$ coincide. Hence
\begin{equation} \label{Inn-PGSp}
     \Inn(\GSp_{2g}(\F_\ell)) \cong \PGSp_{2g}(\F_\ell).
\end{equation}

In addition to inner automorphisms, $\GSp_{2g}(\F_\ell)$ also has radial automorphisms, which we now define in this setting (cf.\ \cite{MR567969}). Let $k\geq 0$ be an integer such that $\gcd(2k+1,\ell-1) = 1$. Then the map
\begin{align*}
\chi_k\colon \GSp_{2g}(\F_\ell) &\to \GSp_{2g}(\F_\ell) \\
\gamma &\mapsto (\mult (\gamma))^k \gamma
\end{align*}
is a \emph{radial automorphism} of $\GSp_{2g}(\F_\ell)$. The set of all radial automorphisms of $\GSp_{2g}(\F_\ell)$ forms a group, which we denote by $\Rad(\GSp_{2g}(\F_\ell))$. It is straightforward to check that the order of $\Rad(\GSp_{2g}(\F_\ell))$ is given by 
\begin{equation} \label{ord-rad}  \abs{\Rad(\GSp_{2g}(\F_\ell))} = | \set{k \in \Z: 0 \leq k < \ell - 1 \text{ and } \gcd(2k+1,\ell-1) = 1} | = 2 \phi(\ell-1)
\end{equation}
where $\phi$ denotes Euler's totient function. The next lemma shows that the inner and radial automorphisms of $\GSp_{2g}(\F_\ell)$ give rise to all automorphisms of $\GSp_{2g}(\F_\ell)$.

\begin{lemma}\label{Aut-GSp} The automorphism group of $\GSp_{2g}(\F_\ell)$ is the internal direct product
\[ \Aut(\GSp_{2g}(\F_\ell)) = \Inn(\GSp_{2g}(\F_\ell)) \times \Rad(\GSp_{2g}(\F_\ell)). \]
\end{lemma}
\begin{proof} Radial automorphisms  of $\GSp_{2g}(\F_\ell)$ act trivially on $\Sp_{2g}(\F_\ell)$. Moreover, by Lemma \ref{L:GSpProperties}(\ref{L:centers}), the only inner automorphism of $\GSp_{2g}(\F_\ell)$ that acts trivially on $\Sp_{2g}(\F_\ell)$ is the identity. Thus
\[ \Inn(\GSp_{2g}(\F_\ell)) \cap \Rad(\GSp_{2g}(\F_\ell)) = \set{\id}. \]
Furthermore, we note that inner and radial automorphisms of $\GSp_{2g}(\F_\ell)$ commute. Indeed, for any $\chi_k \in \Rad(\GSp_{2g}(\F_\ell))$ and $\beta,\gamma \in \GSp_{2g}(\F_\ell)$, we have
\[
\chi_k(\beta \gamma \beta^{-1})
= \mult(\beta \gamma \beta^{-1})^k \beta \gamma \beta^{-1}
= \beta \left(\mult(\gamma)^k\right) \gamma \beta^{-1}
= \beta \chi_k(\gamma) \beta^{-1}.
\]
To establish the claimed result, it remains only to show that the inclusion
\begin{equation} \label{aut-inc} \Rad(\GSp_{2g}(\F_\ell)) \times  \Inn(\GSp_{2g}(\F_\ell)) \subseteq \Aut(\GSp_{2g}(\F_\ell)) \end{equation}
is an equality. 

We proceed by considering orders. By  (\ref{GSp-semi}) and Proposition \ref{aut-semi}, there exists a homomorphism
\[ \Phi\colon \Aut(\GSp_{2g}(\F_\ell)) \to \Aut(\Sp_{2g}(\F_\ell)) \times \Aut(\iota(\F_\ell^\times)) \]
with kernel
\[ \ker \Phi \cong Z^1(\iota(\F_\ell^\times), Z(\Sp_{2g}(\F_\ell))) \cong Z^1(\F_\ell^\times, \set{\pm I}).\]
Since $\set{\pm I}$ is abelian and $\F_{\ell}^{\times}$ is cyclic of even order, 
\[ Z^1(\F_\ell^\times, \set{\pm I})=\text{Hom}(\F_\ell^\times, \set{\pm I}) \cong \set{\pm I}. \]
Hence $\abs{\ker \Phi} = 2$, and by the first isomorphism theorem and Lemma \ref{L:GSpProperties}(\ref{Aut-Sp}),
\[
\abs{\Aut(\GSp_{2g}(\F_\ell))} 
= 2 \abs{\im \Phi} 
\leq 2 \abs{\Aut(\Sp_{2g}(\F_\ell))} \cdot  \abs{ \Aut(\iota(\F_\ell^\times))}
= 2 \abs{\PGSp_{2g}(\F_\ell)} \phi(\ell-1).
\]
Thus, by (\ref{Inn-PGSp}) and (\ref{ord-rad}), we can conclude that (\ref{aut-inc}) is an equality.
\end{proof}

Next, we examine the subgroup $\Aut(\GSp_{2g}(\F_\ell))$  of multiplier-preserving automorphisms,
\begin{equation} \label{def-aut-mult}
\Aut_{\mult}(\GSp_{2g}(\F_\ell)) \coloneqq  \set{\sigma \in \Aut(\GSp_{2g}(\F_\ell)) : \mult \circ \sigma = \mult }.
\end{equation}
We proceed to provide an explicit description of this group  and the action of its elements on the characteristic polynomials $P_\gamma(t) \coloneqq \det\left(tI-\gamma\right)$ for each $\gamma\in \GSp_{2g}(\F_\ell)$.

\begin{lemma}\label{Aut-mult} \indent
\begin{enumerate}  
    \item \label{Aut-mult-1} The group defined in (\ref{def-aut-mult}) is the internal direct product
\[ \Aut_{\mult}(\GSp_{2g}(\F_\ell)) = \Inn{(\GSp_{2g}(\F_\ell))}\times \{\id, \chi_{\frac{\ell-1}{2}} \}. \] 

\item \label{chrpoly-GSp-mult} If $\sigma \in \Aut_{\mult}(\GSp_{2g}(\F_\ell))$, then for each $\gamma \in \GSp_{2g}(\F_\ell)$,
\[ P_{\sigma(\gamma)}(t) \in \set{P_{\gamma}(t), P_{\gamma}(-t)}. \]
\end{enumerate}
\end{lemma}
\begin{proof}
(1) Since the multiplier is invariant under conjugation by matrices in $\GSp_{2g}(\F_\ell)$,
\[ \Inn(\GSp_{2g}(\F_\ell)) \subseteq  \Aut_{\mult}(\GSp_{2g}(\F_\ell)). \]
Let $k$ be an integer such that $0\leq k<\ell-1$ and $\gcd(2k+1, \ell-1)=1$. Suppose that $\chi_k \in \Rad(\GSp_{2g}(\F_\ell)) \cap \Aut_{\mult}(\GSp_{2g}(\F_\ell))$. Then for each $\gamma \in \GSp_{2g}(\F_\ell)$, 
\[ \mult(\gamma) = \mult( \chi_k(\gamma)) = \mult(\mult(\gamma)^k \gamma) = \mult(\gamma)^{2k+1}.\]
Hence 
for each $\gamma \in \GSp_{2g}(\F_\ell)$,
\begin{equation} \label{mod-2k} \mult(\gamma)^{2k}=1. 
\end{equation}
By the surjectivity of the multiplier map, we may take $\gamma$ so that $\mult(\gamma)$ is a generator in $\F_{\ell}^{\times}$.  
Then (\ref{mod-2k}) implies that $k \in \set{0,\frac{\ell-1}{2}}$. Finally, it is clear that $\sigma \circ \chi_{\frac{\ell-1}{2}} \in  \Aut_{\mult}(\GSp_{2g}(\F_\ell))$ for any $\sigma \in\Inn(\GSp_{2g}(\F_\ell))$. So by Lemma \ref{Aut-GSp}, we get \[ \Aut_{\mult}(\GSp_{2g}(\F_\ell)) = \Inn{(\GSp_{2g}(\F_\ell))}\times \{\id, \chi_{\frac{\ell-1}{2}} \}. \]

\eqref{chrpoly-GSp-mult} Let $\sigma \in \Aut_{\mult}(\GSp_{2g}(\F_\ell))$. From part (\ref{Aut-mult-1}) of this lemma,  $\sigma$ is either $\tau$ or $\tau \circ \chi_{\frac{\ell-1}{2}}$ for some $\tau \in \Inn(\GSp_{2g}(\F_\ell))$. Since the characteristic polynomial of a matrix is preserved under conjugation, we only need to consider the situation when $\sigma = \chi_{\frac{\ell-1}{2}}$. For any $\gamma \in \GSp_{2g}(\F_\ell)$, note that
\[
\chi_{\frac{\ell-1}{2}}(\gamma) = (\mult (\gamma))^{\frac{\ell-1}{2}} \gamma \ \in \set{\gamma, -\gamma}.
\]
In particular, $P_{\chi_{\frac{\ell-1}{2}}(\gamma)}(t) \in \{P_{\gamma}(t), P_{-\gamma}(t)\}$. We conclude the proof by observing that 
\[
P_{-\gamma}(t) = \det(tI+\gamma)=(-1)^{2g}\det(-tI-\gamma)= P_{\gamma}(-t).  \qedhere
\]
\end{proof}

\subsubsection{Automorphisms of $\GSp_{2g}(\F_\ell)/\set{\pm I}$} 

In this section, we perform a similar analysis as in the previous subsection for the quotient group $\GSp_{2g}(\F_\ell)/\set{\pm I}$. We will demonstrate that every automorphism of $\GSp_{2g}(\F_\ell)/\set{\pm I}$ can be obtained as the image of an automorphism of $\GSp_{2g}(\F_\ell)$. 
More precisely, it follows from Lemma \ref{Aut-GSp} that $\set{\pm I}$ is a characteristic subgroup  of $\GSp_{2g}(\F_\ell)$, and we aim to show that the map
\[ \Psi\colon \Aut(\GSp_{2g}(\F_\ell)) \to \Aut(\GSp_{2g}(\F_\ell)/\set{\pm I}). \]
is surjective. To do so, we first consider the kernel of $\Psi$.

\begin{lemma} \label{ker-Psi}
We have that $\ker \Psi = \{\id,\chi_{\frac{\ell-1}{2}}\}$.
\end{lemma}
\begin{proof} For any automorphism $\phi \in \Aut(\GSp_{2g}(\F_\ell))$,
\begin{equation} \label{quo-ker}
    \phi \in \ker \Psi
    \quad \iff \quad 
    \phi(\gamma) \in \set{\gamma, -\gamma} \text{ for each $\gamma \in \GSp_{2g}(\F_\ell)$}.
\end{equation}
Observe that $\mult(\phi(\gamma))=\mult(\gamma)$ for each $\gamma \in \GSp_{2g}(\F_\ell)$, so $\phi\in \Aut_{\mult}(\GSp_{2g}(\F_\ell))$. Hence 
\[
\ker \Psi \subset \Aut_{\mult}( \GSp_{2g}(\F_\ell)). 
\]
By Lemma \ref{Aut-mult}(\ref{Aut-mult-1}), we see that \[
\Rad(\GSp_{2g}(\F_\ell)) \cap \ker \Psi = \{\id,\chi_{\frac{\ell-1}{2}} \} 
\subseteq \ker \Psi. 
\]
It suffices to establish the reverse inclusion.

Let $\tau_{\gamma_0} \in \Inn(\GSp_{2g}(\F_\ell))$ denote conjugation by a matrix $\gamma_0 \in \GSp_{2g}(\F_\ell)$. Note that
\[ \tau_{\gamma_0} \circ \chi_{\frac{\ell-1}{2}} \in \ker \Psi \quad \Longleftrightarrow \quad \tau_{\gamma_0} = \tau_{\gamma_0} \circ \chi_{\frac{\ell-1}{2}} \circ \chi_{\frac{\ell-1}{2}} \in \ker \Psi.\]
Thus it suffices to show that if $\tau_{\gamma_0} \in \ker \Psi$, then $\tau_{\gamma_0}=\id$.  We denote the image of $\gamma_0$ in $\GSp_{2g}(\F_\ell)/\{\pm I\}$ by $ \bar{\gamma}_0$. From (\ref{quo-ker}), we note that  if $\tau_{\gamma_0} \in \ker \Psi$, then
\[ \bar{\gamma}_0 \bar{\gamma} = \bar{\gamma} \bar{\gamma}_0 \quad \text{for each } \gamma \in \GSp_{2g}(\F_\ell). \]
Hence $\bar{\gamma}_0$ is contained in the center of $\GSp_{2g}(\F_\ell)/\set{\pm I}$. It follows from Lemma \ref{L:GSpProperties}\eqref{L:centers} that the center of $\GSp_{2g}(\F_\ell)/\set{\pm I}$ is $R_{2g}(\F_\ell)/\set{\pm I}$. Hence $\tau_{\gamma_0} = \id$.
\end{proof}

\begin{lemma} \label{surj-aut} The group homomorphism 
\[ \Psi\colon \Aut(\GSp_{2g}(\F_\ell)) \to \Aut(\GSp_{2g}(\F_\ell)/\set{\pm I}) \]
induced by the quotient map $\GSp_{2g}(\F_\ell) \to \GSp_{2g}(\F_\ell)/\set{\pm I}$ is surjective.
\end{lemma}
\begin{proof} 
By the first isomorphism theorem, (\ref{Inn-PGSp}), (\ref{ord-rad}),  Lemma \ref{Aut-GSp}, and Lemma \ref{ker-Psi}, we  have that
\begin{equation*} \label{image-card}
\abs{\im \Psi} = \frac{1}{2}\abs{\Aut(\GSp_{2g}(\F_\ell))} = \abs{\PGSp_{2g}(\F_\ell)} \phi(\ell-1).
\end{equation*}
It suffices now to determine the order of $\Aut(\GSp_{2g}(\F_\ell)/\set{\pm I})$. By Lemma \ref{L:GSpProperties}\eqref{GSppm}, we have that
\[ \GSp_{2g}(\F_\ell) / \set{\pm I} \cong \PSp_{2g}(\F_\ell) \rtimes \iota(\F_\ell^\times). \]
From Lemma \ref{L:GSpProperties}  \eqref{L:GSpCharSub}, we have that $\PSp_{2g}(\F_\ell)$ is a characteristic subgroup of $\GSp_{2g}(\F_\ell)/\{\pm I\}$. Hence by Proposition \ref{aut-semi}, there exists a homomorphism
\[ \Phi' \colon \Aut(\GSp_{2g}(\F_\ell) / \set{\pm I}) \to \Aut(\PSp_{2g}(\F_\ell)) \times \Aut(\iota(\F_\ell^\times)). \]
Using Lemma \ref{L:GSpProperties}\eqref{L:centers}, we observe that the kernel of $\Phi'$ is
\[
\ker(\Phi') = Z^1(\iota(\F_\ell^\times), Z(\PSp_{2g}(\F_\ell)))\cong \text{Hom}(\F_{\ell}^{\times}, \{\overline{I}\})=\{\id\}.
\]
Thus $\Phi'$ is  injective.  Note that by Lemma \ref{L:GSpProperties}(\ref{Aut-Sp})
\[
\Aut(\PSp_{2g}(\F_\ell)) \cong \PGSp_{2g}(\F_\ell) \quad \text{and} \quad \Aut(\iota(\F_\ell^\times))  \cong (\Z/(\ell-1)\Z)^\times.
\]
Hence, 
\[
\abs{\Aut(\GSp_{2g}(\F_\ell)/\set{\pm I})}\leq \abs{\Aut(\PSp_{2g}(\F_\ell))}\abs{\Aut(\iota(\F_\ell^\times)) }=\abs{\PGSp_{2g}(\F_\ell)} \phi(\ell-1)=\abs{\im \Psi}.
\]
This implies  $\Psi$ is surjective. 
\end{proof}

As before, we consider automorphisms that preserve the multiplier. Observe that for any $\gamma \in \GSp_{2g}(\F_\ell)$, 
\[ \mult(\gamma) = \mult(-\gamma). \]
Thus we may also consider $\mult$ as a well-defined surjective homomorphism
\begin{align*}
\mult\colon \GSp_{2g}(\F_\ell)/\set{\pm I} \twoheadrightarrow \F_\ell^\times.
\end{align*}
With this in mind, we consider the group 
\[
\Aut_{\mult}(\GSp_{2g}(\F_\ell)/\set{\pm I}) \coloneqq  \set{\sigma \in \Aut(\GSp_{2g}(\F_\ell)/\set{\pm I}) :  \mult \circ \sigma = \mult}.
\]

For $\sigma \in \Aut_{\mult}(\GSp_{2g}(\F_\ell)/\set{\pm I})$ and $\gamma 
\in \GSp_{2g}(\F_\ell)/\set{\pm I}$, denote by $\sigma(\gamma)'$ and $\gamma'$  lifts of $ \sigma(\gamma)$ and $\gamma$ in $\GSp_{2g}(\F_\ell)$, respectively. 
Our next result is an analog of Lemma \ref{Aut-mult}\eqref{chrpoly-GSp-mult} in the present setting.

\begin{lemma} \label{chrpoly-mod} If $\sigma \in \Aut_{\mult}(\GSp_{2g}(\F_\ell))/\{\pm I\}$, then for each $(\gamma_1,\gamma_2) \in \Delta_{2g}(\F_\ell)$,
\[
\bar{\gamma}_1 = \sigma(\bar{\gamma}_2)
\quad \implies \quad
P_{\gamma_1}(t) \in \set{P_{\gamma_2}(t), P_{\gamma_2}(-t)}, \]
where $\bar{\gamma}$ denotes the image of $\gamma\in \GSp_{2g}(\F_\ell)$ in $\GSp_{2g}(\F_\ell)/\{\pm I\}$.
\end{lemma}
\begin{proof}
By  Lemma \ref{surj-aut}, the automorphism 
$\sigma \in \Aut_{\mult}(\GSp_{2g}(\F_\ell)/\set{\pm I})$ can be lifted to $\sigma'\in \Aut_{\mult}(\GSp_{2g}(\F_\ell))$. Then $\gamma_1 = \pm \sigma(\overline{\gamma}_2)'=\pm \sigma'(\gamma_2)$.  Hence by Lemma \ref{Aut-mult}(\ref{chrpoly-GSp-mult}),
\[
P_{\gamma_1}(t)\in \{P_{\sigma'(\gamma_2)}(t), P_{-\sigma'(\gamma_2)}(t)\}=\{P_{\sigma'(\gamma_2)}(t), P_{\sigma'(\gamma_2)}(-t)\}= 
\{P_{\gamma_2}(t), P_{\gamma_2}(-t)\}.
\]
The claim follows.
\end{proof}

\subsection{Subgroups of \texorpdfstring{$\Delta_{2g}(\F_\ell)$}{Delta2gFell} with surjective projections}

We now prove our main group-theoretic proposition, which classifies subgroups of $\Delta_{2g}(\F_\ell)$ that surject onto both factors. Our proof considers three cases arising from Goursat's lemma, and utilizes Lemma \ref{Aut-mult}(\ref{Aut-mult-1}) and Lemma \ref{surj-aut}.

\begin{proposition} \label{prop-class}  Let $G$ be a subgroup of $\Delta_{2g}(\F_\ell)$. 
If $ \pi_i(G) = \GSp_{2g}(\F_\ell)$ holds for each $i \in \set{1,2}$, then one of the following holds: 
\begin{enumerate}
    \item \label{class-1} $G = \set{(\gamma,\phi(\gamma)) \in \Delta_{2g}(\F_\ell) : \gamma \in \GSp_{2g}(\F_\ell)}$ for some $\phi \in \Aut_{\mult}(\GSp_{2g}(\F_\ell))$, 
    \item \label{class-2}  $G = \set{(\gamma_1,\gamma_2) \in \Delta_{2g}(\F_\ell) :
    \bar{\gamma}_1 = \psi(\bar{\gamma}_2)}$ for some $\psi \in \Aut_{\mult}(\GSp_{2g}(\F_\ell) / \set{\pm I})$, 
    \item \label{class-3} $G = \Delta_{2g}(\F_\ell)$.
\end{enumerate}
\end{proposition}
\begin{proof}

By Lemma \ref{Goursat}, there exists a group $Q$ and surjective homomorphisms
\[
\phi_1\colon \GSp_{2g}(\F_\ell) \twoheadrightarrow Q 
\quad \text{and} \quad
\phi_2\colon \GSp_{2g}(\F_\ell) \twoheadrightarrow Q
\]
such that $G$ is the fiber product
\begin{equation} G = \GSp_{2g}(\F_\ell) \times_{(\phi_1,\phi_2)}  \GSp_{2g}(\F_\ell). \label{Gfiber} \end{equation}
Note that $\ker \phi_1 \times \set{I}
\subseteq G$. Since  $G \subseteq \Delta_{2g}(\F_\ell)$, for each $k_1 \in \ker \phi_1$, we have that $\mult(k_1) = \mult(I) = 1$. Hence $\ker \phi_1 \subseteq \Sp_{2g}(\F_\ell)$, and  it follows from Lemma \ref{L:GSpProperties}(\ref{normal-Sp}) that
\[ \ker \phi_1 \in \set{ \set{I}, \set{\pm I},\Sp_{2g}(\F_\ell)}.\]
As $Q \cong \GSp_{2g}(\F_\ell)/\ker\phi_1$, we deduce that up to isomorphism,
\[
Q \in \set{  \GSp_{2g}(\F_\ell), \GSp_{2g}(\F_\ell)/\set{\pm I},\F_\ell^\times }.
\]
The three cases give rise to the three possibilities in the statement of the proposition, as we now show.

\emph{Case 1.} Assume that $Q \cong \GSp_{2g}(\F_\ell)$. Then for each $i \in \set{1,2}$,  $\phi_i$ is an isomorphism. We  consider the automorphism $\phi_2^{-1} \circ \phi_1 \in \Aut(\GSp_{2g}(\F_\ell))$. For each $\gamma \in \GSp_{2g}(\F_\ell)$, note that 
\[ \phi_1(\gamma) = \phi_2((\phi_2^{-1} \circ \phi_1)(\gamma)).\]
Hence by (\ref{Gfiber}),
\[
( \gamma, (\phi_2^{-1} \circ \phi_1)(\gamma) ) \in G.
\]
Thus as $G \subseteq \Delta_{2g}(\F_\ell)$,
\[ \phi_2^{-1} \circ \phi_1 \in \Aut_{\mult}(\GSp_{2g}(\F_\ell)).\]
Now we set $H  \coloneqq  \set{(\gamma, (\phi_2^{-1} \circ \phi_1)(\gamma)) : \gamma \in \GSp_{2g}(\F_\ell)} \subseteq G$. Then,  $G = H$, since by  (\ref{order-Delta}),
\[ \abs{G} = \frac{\abs{\GSp_{2g}(\F_\ell)}^2 }{\abs{\GSp_{2g}(\F_\ell)}} = \abs{\GSp_{2g}(\F_\ell)} = \abs{H}. \]

\textit{Case 2.} Assume that $Q \cong \GSp_{2g}(\F_\ell) / \set{\pm I}$. Then, for  each $i \in \set{1,2}$,  $\phi_i:\GSp_{2g}(\F_\ell) \twoheadrightarrow  Q$ is a surjective homomorphism with  $\ker \phi_i = \set{\pm I}$. Thus $\phi_i$ induces the isomorphism
\[ \tilde{\phi}_i \colon \GSp_{2g}(\F_\ell)/\set{\pm I} \to  Q  \]
defined such that  for each $\gamma \in \GSp_{2g}(\F_\ell)$,   $\tilde{\phi}_i(\bar{\gamma}) = \phi_i(\gamma)$, where $\bar{\gamma}$ denotes the coset of $\gamma$ in $\GSp_{2g}(\F_\ell)/\set{\pm I}$. For each $\gamma \in \GSp_{2g}(\F_\ell)$, note that
\[ \tilde{\phi}_1(\bar{\gamma}) = \tilde{\phi}_2((\tilde{\phi}_2^{-1} \circ \tilde{\phi}_1)(\bar{\gamma})). \]
Hence $(\bar{\gamma},(\tilde{\phi}_2^{-1} \circ \tilde{\phi}_1)(\bar{\gamma}))$ lies in the image of $G$ under the projection 
\[ \Delta_{2g}(\F_\ell) \to \GSp_{2g}(\F_\ell)/\set{\pm I} \times_{(\mult,\mult)} \GSp_{2g}(\F_\ell)/\set{\pm I}.\]
Since $G \subseteq \Delta_{2g}(\F_\ell)$, it follows that
\[ \tilde{\phi}_1^{-1} \circ \tilde{\phi}_2 \in \Aut_{\mult}(\GSp_{2g}(\F_\ell) / \set{\pm I}). \]
Set $H \coloneqq  \{(\gamma_1,\gamma_2) \in \Delta_{2g}(\F_\ell) : \bar{\gamma}_1 = (\tilde{\phi}_2^{-1} \circ \tilde{\phi}_1)(\bar{\gamma}_2)\} \subseteq G$.
We see that $G = H$ since by \eqref{order-Delta}
\[ \abs{G} = \frac{\abs{\GSp_{2g}(\F_\ell)}^2}{\abs{\GSp_{2g}(\F_\ell)/\set{\pm I}}} = 2 \abs{\GSp_{2g}(\F_\ell)} = \abs{H}. \]

\textit{Case 3.} Assume that $Q \cong \F_\ell^\times$. Then by \eqref{order-Delta}, the order of $G$ coincides with that of $\Delta_{2g}(\F_\ell)$, and hence $G = \Delta_{2g}(\F_\ell)$.
\end{proof}

From this classification, we can derive two corollaries that characterize when a subgroup $G$ of $\Delta_{2g}(\F_\ell)$ is equal to the full group $\Delta_{2g}(\F_\ell)$, by examining whether there exists an element in $G$ whose characteristic polynomial or trace satisfies a certain condition. 

Recall that $\pr_i:\Delta_{2g}(\F_\ell) \twoheadrightarrow \GSp_{2g}(\F_\ell)$ is the projection map onto the $i$-th component.

\begin{corollary} \label{corr-class-chrpoly} Assume $\ell \geq 5$.  A subgroup $G \subseteq \Delta_{2g}(\F_\ell)$ satisfies $G = \Delta_{2g}(\F_\ell)$ if and only if
\begin{enumerate}
\item For each $i \in \{1,2\}$, we have that $\pr_i(G) = \GSp_{2g}(\F_\ell)$ and
\item There exists an element $(\gamma_1,\gamma_2) \in G$ such that $P_{\gamma_1}(t) \not\in \set{P_{\gamma_2}(t), P_{\gamma_2}(-t)}$.
\end{enumerate}  
\end{corollary} 
\begin{proof} The forward direction is clear. For the reverse direction, Lemma \ref{Aut-mult}(\ref{chrpoly-GSp-mult}) and  Lemma \ref{chrpoly-mod} rule out cases (\ref{class-1}) and (\ref{class-2}) in Proposition \ref{prop-class}, leaving only case (\ref{class-3}), that $G = \Delta_{2g}(\F_\ell)$.
\end{proof}

\begin{corollary} \label{corr-class} Assume $\ell \geq 5$.  A subgroup $G \subseteq \Delta_{2g}(\F_\ell)$ satisfies $G = \Delta_{2g}(\F_\ell)$ if and only if
\begin{enumerate}
\item For each $i \in \{1,2\}$, we have that $\pr_i(G) = \GSp_{2g}(\F_\ell)$ and
\item There exists an element $(\gamma_1,\gamma_2) \in G$ such that $\tr \gamma_1 \not\in \set{\tr \gamma_2,-\tr \gamma_2}$.
\end{enumerate}  
\end{corollary} 
\begin{proof}  The forward direction is clear. The reverse direction follows from Corollary \ref{corr-class-chrpoly}.
\end{proof}

\subsection{Subgroups of \texorpdfstring{$\Delta_{g_1, g_2}(\F_\ell)$}{Deltag1g2(Fell)} with surjective projections}

We now turn our attention to the case where $n = 2$ and $g_1, g_2$ are distinct positive integers. In particular, we prove a group-theoretic lemma that classifies subgroups of $\Delta_{g_1, g_2}(\F_\ell)$ with  surjective projections  onto both factors. This lemma is essential for proving the effective open image theorem for a product of two abelian varieties of different dimensions (Theorem \ref{main-thrm-2}).

\begin{lemma}\label{group-lemma-2} Assume that $\ell\geq 5$ and  $g_1 \neq g_2$. If a subgroup $G$ of $\Delta_{g_1,g_2}(\F_{\ell}) $ surjects onto both $\GSp_{2g_1}(\F_{\ell})$ and $\GSp_{2g_2}(\F_{\ell})$ via the natural projection maps, then $G = \Delta_{g_1,g_2}(\F_\ell)$.
\end{lemma}
\begin{proof} By Lemma \ref{Goursat}, we may write
\[
G = \GSp_{2g_1}(\F_\ell) \times_{(\phi_1,\phi_2)} \GSp_{2g_2}(\F_\ell)
\]
for some $\phi_1, \phi_2$ such that
\[ \GSp_{2g_1}(\F_\ell) / \ker \phi_1 \cong  \GSp_{2g_2}(\F_\ell) / \ker \phi_2. \]
Similar to the argument in Proposition \ref{prop-class}, we have that
\[
\ker \phi_1 \subseteq \Sp_{2g_1}(\F_\ell)
\quad \text{and} \quad
\ker \phi_2 \subseteq \Sp_{2g_2}(\F_\ell).
\]
Thus for $i \in \{1,2\}$,
\[ \ker \phi_i \in \{ \{I\}, \{\pm I \}, \Sp_{2g_i}(\F_\ell) \}. \]
Hence if $\ker \phi_i \neq \Sp_{2g_i}(\F_\ell)$, then
\[
[\GSp_{2g_i}(\F_\ell) : \ker \phi_i] = |\GSp_{2g_i}(\F_\ell)|
\quad \text{or} \quad
[\GSp_{2g_i}(\F_\ell) : \ker \phi_i] = \tfrac{1}{2} |\GSp_{2g_i}(\F_\ell)|.
\]
However, since $g_1 \neq g_2$, it follows from \eqref{E:gsp-order} that
\[
\{ |\GSp_{2g_1}(\F_\ell)|, \tfrac{1}{2} |\GSp_{2g_1}(\F_\ell)| \} \cap \{ |\GSp_{2g_2}(\F_\ell)|, \tfrac{1}{2} |\GSp_{2g_2}(\F_\ell)| \} = \emptyset.
\]
Therefore, we have $\ker \phi_i = \Sp_{2g_i}(\F_\ell)$ for $i \in \{ 1, 2\}$. Hence $G = \Delta_{g_1,g_2}(\F_\ell)$.
\end{proof}

\section{An effective version of Faltings's isogeny theorem}\label{S4}

\subsection{An extension of Serre's lemma}

We begin by proving Proposition \ref{lemma-Serre-isogeny}, which is an extension of a lemma of Serre on $\ell$-adic representations \cite[Note 632.6]{MR3185222},  \cite[Theorem 4.7]{MR3502938}.

To set the notation, let $\ell$ be a prime number, $g$ be a positive integer, $\Gamma$ be a profinite group, and 
\[ \rho_{1} \colon \Gamma \to \GSp_{2g}(\Z_\ell) \quad \text{and}\quad \rho_{2} \colon \Gamma \to \GSp_{2g}(\Z_\ell)\]
be continuous homomorphisms. Assume that the image, denoted $\overline{\Gamma}$, of the product map
\begin{align*}\rho_1\times \rho_2\colon\Gamma &\to  \GSp_{2g}(\Z_\ell) \times \GSp_{2g}(\Z_\ell) 
\end{align*}
is a subgroup of $\Delta_{2g}(\Z_\ell)$. Let $M$ denote the $\Z_\ell$-algebra of all $2g \times 2g$ matrices with  entries in $\Z_\ell$   and  let $A$ denote the subalgebra of $M\times M$ generated by $\overline{\Gamma}$. Since $\overline{\Gamma}$ is a group, $A$ equals the $\Z_\ell$-span of $\overline{\Gamma}$ inside $M \times M$. We consider the  map
 \begin{align*}
T\colon M \times M &\to  \Z_{\ell} \\
(x_1, x_2) &\mapsto \tr(x_1)-\tr( x_2),
\end{align*}
which one can verify is a $\Z_\ell$-module homomorphism.

We now state our target proposition, which is crucial for our proof of Theorem \ref{effective-Faltings}.

\begin{proposition} \label{lemma-Serre-isogeny}  We keep the notation above. Assume that $\ell\nmid 2g$ and that $\overline{\Gamma}$ is an open subgroup of $\Delta_{2g}(\Z_\ell)$. Then there exists a finite quotient $G$ of $\Gamma$ and a subset $C_{\pm} \subseteq G$ with the properties that
\begin{enumerate}
\item the set $C_{\pm}$ is nonempty and is closed under conjugation in $G$,
\item the cardinality of $G$ is at most $(\ell^{8g^2} - 1)^2$, and 
\item \label{Cond3} if the image of $\gamma \in \Gamma$  in $G$ belongs to $C_{\pm}$, then $\tr \rho_1(\gamma) \neq \pm \tr \rho_2(\gamma)$.
\end{enumerate}
\end{proposition}

\begin{remark}
An earlier version of this article presented the above proposition in a more general, but incorrect, form. We are grateful to the referee for identifying this error and bringing it to our attention. While it seems possible to frame the proposition more generally than in its current form (by replacing $\GSp_{2g}(\Z_\ell)$ with $\GL_{r}(\Z_\ell)$ for any $r \geq 2$ and weakening the condition that the image of the product map is open), we do not explore this further, as it is not necessary in this paper.
\end{remark}

The rest of this subsection is devoted to proving Proposition \ref{lemma-Serre-isogeny}. Our proof is inspired by the proof of Serre's lemma that appears in \cite[Theorem 4.7]{MR3502938}. However, the situation here is complicated by conclusion \eqref{Cond3} of Proposition \ref{lemma-Serre-isogeny}, namely that for all $\gamma \in \Gamma$ whose image in $G$ belongs to $C_{\pm}$, we have that $\tr \rho_1(\gamma)$ is distinct from {both} $\tr \rho_2(\gamma)$ and $-\tr \rho_2(\gamma)$ simultaneously.

\begin{proof} We begin by setting up the notation to define certain distinguished sets $C_+, C_-, $ and $C_{\pm}$ associated with $A$. Consider the naturally defined map
\begin{equation} \label{GammaAellA}  \overline{\Gamma} \hookrightarrow A \twoheadrightarrow A/\ell A. \end{equation}
The image of \eqref{GammaAellA} is a group in the $\Z_\ell$-algebra $A/\ell A$, as the map above preserves the multiplicative structure. We denote by  $\Gamma_1\trianglelefteq  \overline{\Gamma}$ the kernel of the group homomorphism. It is convenient to view $\overline{\Gamma}/\Gamma_1$ as a subset of $A/\ell A$ in light of the natural injective map $\overline{\Gamma}/\Gamma_1\hookrightarrow A/\ell A$ of sets. Our setup also involves the following commutative diagram:
\begin{equation*}
\begin{tikzcd}
 \overline{\Gamma} \ar[r, hook] \ar[d, two heads, "\pi_1"] & A \ar[r] \ar[d, two heads] & A/\ell A \ar[d, "\bar{\lambda}_+"]\\
 \overline{\Gamma}/\Gamma_1 \ar[r, hook] & A/\ell A \ar[r, "\bar{\lambda}_+"] & \Z/\ell\Z
\end{tikzcd}
\end{equation*}
In the diagram above, the first two vertical surjective maps are canonical projections as group and $\Z_\ell$-module homomorphisms, respectively. 
The maps in the top row are defined as in (\ref{GammaAellA}). The first map in the second row is the inclusion.  Recalling that $A$ is the $\Z_\ell$-span of $\overline{\Gamma}$, we have that $\tr \rho_1(\gamma) \equiv \tr \rho_2 (\gamma) \pmod{\ell^m}$ for all $\gamma\in \Gamma$ implies $\tr x_1 \equiv  \tr x_2 \pmod \ell$ for all $(x_1, x_2)\in A$. 
Then, the vertical $\Z/\ell\Z$-module homomorphism $\bar{\lambda}_+\colon A/\ell A\to \Z/\ell\Z$ is    induced by the $\Z_{\ell}$-module homomorphism
$\lambda_+\colon A \to \Z_{\ell}$, defined by
\[
\lambda_+(x_1,x_2) = \ell^{-m}(\tr x_1 - \tr x_2),
\]
where $m$ is the largest nonnegative integer such that the congruence 
\[ \tr \rho_1(\gamma) \equiv \tr \rho_2 (\gamma) \pmod{\ell^m} \]
holds for each $\gamma \in \Gamma$. The existence of such an $m$ follows from the assumption that $\overline{\Gamma}$ is an open subgroup of $\Delta_{2g}(\Z_\ell)$, since no open subgroup of $\Delta_{2g}(\Z_\ell)$ can be contained in the set $\{(\gamma_1,\gamma_2) \in \Delta_{2g}(\Z_\ell) : \tr \gamma_1 = \tr \gamma_2\}$ by index considerations. Moreover,   $\bar{\lambda}_+$ is well-defined because $\lambda_+(\ell A)\subseteq \ell\Z_{\ell}$. Considering $\overline{\Gamma}/\Gamma_1$ as a subset of $A/\ell A$, we define the set
\begin{equation}\label{C-plus}
C_+ \coloneqq  (\overline{\Gamma} / \Gamma_1) \setminus  \ker \bar{\lambda}_+. 
\end{equation}

A similar diagram can be made for the construction  of $C_-$:
\begin{equation*}
\begin{tikzcd}
 \overline{\Gamma} \ar[r, hook] \ar[d, two heads, "\pi_1"] & A \ar[r] \ar[d, two heads]& A/\ell A
 \ar[d, "\bar{\lambda}_-"] \\
 \overline{\Gamma}/\Gamma_1 \ar[r, hook] & A/\ell A \ar[r, "\bar{\lambda}_-"] & \Z/\ell\Z
\end{tikzcd}
\end{equation*}
In this second diagram, the vertical $\Z/\ell\Z$-module homomorphism   $\bar{\lambda}_-\colon A/\ell A\to \Z/\ell\Z$ is   induced by the $\Z_{\ell}$-module homomorphism $\lambda_-\colon A\to \Z_{\ell}$, which itself is defined by
\[ \lambda_-(x_1,x_2) = \tr x_1 + \tr x_2. \]
Again viewing $\overline{\Gamma}/\Gamma_1$ as a subset of $A/\ell A$, we define the set
\begin{equation}\label{C-minus}
C_- \coloneqq  (\overline{\Gamma} / \Gamma_1) \setminus \ker \bar{\lambda}_-. 
\end{equation}

Next, let $A_1$ denote the  $\Z_\ell$-subalgebra of $A$ generated by  $\Gamma_1$.  Analogously to \eqref{GammaAellA}, we consider the map
\begin{equation} \label{GammaAellA1} \Gamma_1 \hookrightarrow A_1 \twoheadrightarrow A_1/\ell A_1 \end{equation}
Similarly as before, the image of \eqref{GammaAellA1} is a group.

Our goal now is to define a  map $\bar{\lambda}_{\pm}\colon A_1/\ell A_1 \to \Z/\ell\Z$, which will play an essential role in proving Proposition \ref{lemma-Serre-isogeny}. To do so, we first establish some properties of the aforementioned objects. With an abuse of notation, we write $\pi_1$ for the map 
\[ \pi_1 \colon \Gamma \twoheadrightarrow \overline{\Gamma} \twoheadrightarrow \overline{\Gamma}/\Gamma_1.\]

\begin{lemma}\label{prop-Serre-lemma-0} We keep the notation and assumptions as above.
\begin{enumerate}
    \item \label{lemma-Cplus}
    The set $C_+$ is nonempty and closed under conjugation by elements of $\overline{\Gamma}/\Gamma_1$. Moreover, for each $\gamma \in \Gamma$,
    if $\pi_1(\gamma) \in C_+$, then   $\tr \rho_1(\gamma) \neq \tr \rho_2(\gamma)$.
    \item \label{lemma-Cminus} The set $C_-$ is nonempty and is closed under conjugation by elements of $\overline{\Gamma}/\Gamma_1$. Moreover, for each $\gamma \in \Gamma$, if
   $\pi_1(\gamma) \in C_-$, then $\tr \rho_1(\gamma) \neq -\tr \rho_2(\gamma)$.
   \item 
   \label{lemma-Gamma1-minus} For each $\gamma \in \Gamma$, if the image of $\gamma$ in $\overline{\Gamma}$ is contained in  $\Gamma_1$, then  $\tr \rho_1(\gamma) \neq -\tr \rho_2(\gamma)$.
   \item \label{lemma-Gamma1-plus} There exists $\gamma \in \Gamma$ such that  the image of $\gamma$ in $\overline{\Gamma}$ is contained in  $\Gamma_1$ and  $\tr \rho_1(\gamma) \neq \tr \rho_2(\gamma)$.
   \end{enumerate}
\end{lemma}
\begin{proof}
\begin{enumerate}
    \item[(1)] 
    From the definition of $\lambda_+$, there exists an element $\gamma\in \Gamma$ such that 
\[
\tr \rho_1(\gamma)\not \equiv \tr \rho_2(\gamma) \pmod{\ell^{m+1}}. 
\]
For such $\gamma$, we have that $\bar{\lambda}_+(\pi_1(\gamma))\neq 0$.
Consequently, $C_+$ is nonempty.  Since the trace is invariant under conjugation,  $C_+$ is invariant under conjugation by elements of $\overline{\Gamma}/\Gamma_1$. Finally, if $\pi_1(\gamma)\in C_+$, then $\bar{\lambda}_+(\pi_1(\gamma))\neq 0$, so
\[
\tr \rho_1(\gamma)\not \equiv \tr \rho_2(\gamma) \pmod{\ell^{m+1}}, 
\]
which implies that $\tr \rho_1(\gamma) \neq \tr \rho_2(\gamma)$.
    \item[(2)] Let $\id$ denote the identity element of the group $\Gamma$. By the assumption that $\ell \nmid 2g$, we have
    \[
    \tr \rho_1(\id) \not\equiv - \tr \rho_2(\id) \pmod{\ell}.
    \]
    Thus $\bar{\lambda}_-(\pi_1(\id)) \neq 0$, and hence $C_-$ is nonempty. The remaining claims in \eqref{lemma-Cminus} follow similarly to the proof of  \eqref{lemma-Cplus}.
    \item[(3)] Let $\gamma\in \Gamma$ 
 be such that the image of $\gamma$ in $\overline{\Gamma}$ is contained in  $\Gamma_1$. Then, this image  in $A/\ell A$ is the multiplicative identity element $(I_{2g},I_{2g}) + \ell A$. Since $\ell\nmid 2g$, we get
\[ \bar{\lambda}_-(\pi_1(\gamma)) \equiv 4g \not\equiv 0   \pmod{\ell}.\] 
In particular, $\pi_1(\gamma) \in C_-$. Thus the claim follows from the definition of $C_-$ and (\ref{lemma-Cminus}).

\item[(4)] We prove this part with an explicit construction.  By assumption, $\overline{\Gamma}$ is an open subgroup of $\Delta_{2g}(\Z_\ell)$. Thus, we may let $r$ be a positive integer large enough such that $\overline{\Gamma} = \pi^{-1}(\overline{\Gamma}(\ell^r))$, where $\pi \colon \Delta_{2g}(\Z_\ell) \to \Delta_{2g}(\Z/\ell^r\Z)$ denotes the reduction modulo $\ell^r$ map and $\overline{\Gamma}(\ell^r)$ denotes the image of $\overline{\Gamma}$ under this map. Then,
\[K \coloneqq \ker( \Delta_{2g}(\Z_\ell) \to \Delta_{2g}(\Z/\ell^r\Z)) \subseteq \overline{\Gamma}.\]
For a pair of matrices $(x_1,x_2) \in M \times M$, we have that $(x_1,x_2) \in K$ if and only if $x_1,x_2 \in \GSp_{2g}(\Z_\ell)$, $\mult x_1 = \mult x_2$, and the reductions of $x_1$ and $x_2$ are both congruent to $I_{2g}$ modulo $\ell^r$. Checking that these conditions hold, we find that
\begin{align} \label{E:Construction1}
    \left( I_{2g}, I_{2g} + \ell^r \begin{psmallmatrix} 0 & (\ell^r-1) I_{g} \\ 0 & 0 \end{psmallmatrix}\right) &\in K \subseteq A \\
    \label{E:Construction2} \left( I_{2g} + \ell^r \begin{psmallmatrix} \ell^{2r} I_{g} & 0 \\ 0 & 0 \end{psmallmatrix}, I_{2g} + \ell^r \begin{psmallmatrix} 0 & -\ell^r I_{g} \\ I_{g} & 0 \end{psmallmatrix}\right) &\in K \subseteq A,
\end{align}
where $0$ denotes the $g \times g$ zero matrix. Since $(I_{2g},I_{2g}) \in A$ and $A$ is closed under subtraction, we have
\begin{align*} 
    \left(\begin{psmallmatrix} 0 & 0 \\ 0 & 0 \end{psmallmatrix},  \begin{psmallmatrix} 0 & (\ell^{2r}-\ell^r)I_{g} \\ 0 & 0 \end{psmallmatrix}\right) &\in A \\
     \left(\begin{psmallmatrix} \ell^{3r} I_{g} & 0 \\ 0 & 0 \end{psmallmatrix}, \begin{psmallmatrix} 0 & -\ell^{2r} I_{g} \\ \ell^r I_{g} & 0 \end{psmallmatrix}\right) &\in A.
\end{align*}
Thus, their sum 
\[
\left(\begin{psmallmatrix} \ell^{3r} I_{g} & 0 \\ 0 & 0 \end{psmallmatrix}, \begin{psmallmatrix} 0 & -\ell^r I_{g} \\ \ell^r I_{g} & 0 \end{psmallmatrix}\right)
\]
is also in $A$. Similarly as for \eqref{E:Construction1} and \eqref{E:Construction2}, we further verify that
\[
\left(I_{2g} + \ell^r \begin{psmallmatrix} \ell^{3r} I_{g} & 0 \\ 0 & 0 \end{psmallmatrix}, I_{2g} + \ell^r \begin{psmallmatrix} 0 & -\ell^r I_{g} \\ \ell^r I_{g} & 0 \end{psmallmatrix}\right) \in K.
\]
Therefore, since $K \subseteq \overline{\Gamma}$, there exists an element $\gamma \in \Gamma$  such that
\begin{align*}
(\rho_1 \times \rho_2)(\gamma) &= \left(I_{2g} + \ell^r \begin{psmallmatrix} \ell^{3r} I_{g} & 0 \\ 0 & 0 \end{psmallmatrix}, I_{2g} + \ell^r \begin{psmallmatrix} 0 & -\ell^r I_{g} \\ \ell^r I_{g} & 0 \end{psmallmatrix}\right) \\
&= (I_{2g},I_{2g}) + \ell^r \left(\begin{psmallmatrix} \ell^{3r} I_{g} & 0 \\ 0 & 0 \end{psmallmatrix}, \begin{psmallmatrix} 0 & -\ell^r I_{g} \\ \ell^r I_{g} & 0 \end{psmallmatrix}\right).
\end{align*}
The element $\gamma$ satisfies the desired property. Indeed,
\[(\rho_1 \times \rho_2)(\gamma) \in (I_{2g},I_{2g}) + \ell^r A \subseteq (I_{2g},I_{2g}) + \ell A = \Gamma_1,\]
and further,
\[\tr \rho_1 (\gamma) 
= 2g + \ell^{4r} g \quad \text{and} \quad \tr \rho_2 (\gamma) = 2g,\]
so $\tr \rho_1 (\gamma) \neq \tr \rho_2 (\gamma)$. This completes  the proof.\qedhere
\end{enumerate}
\end{proof}

Now we denote by  $\Gamma_2\trianglelefteq  \Gamma_1$ the kernel of the group homomorphism \eqref{GammaAellA1}.
Consider the following commutative diagram, which will be used to construct $C_{\pm}$:
\begin{equation*}
\begin{tikzcd}
   \Gamma_1  \ar[r, hook] \ar[d, two heads, "\pi_2"] & A_1 \ar[r] \ar[d, two heads] & A_1/\ell A_1 \ar[d, "\bar{\lambda}_{\pm}"]\\
  \Gamma_1/\Gamma_2 \ar[r, hook] & A_1/\ell A_1 \ar[r, "\bar{\lambda}_{\pm}"] & \Z/\ell\Z
\end{tikzcd}
\end{equation*}
The first two vertical maps  are projections. The vertical $\Z/\ell\Z$-module homomorphism  $\bar{\lambda}_{\pm} \colon A_1/\ell A_1 \to \Z/\ell\Z$ is induced by $\Z_{\ell}$-module homomorphism
$\lambda_{\pm} \colon A_1 \to \Z_{\ell}$ defined by
\[ {\lambda}_{\pm}(x_1,x_2) = \ell^{-m_{\pm}}(\tr x_1 - \tr x_2),\]
where $m_{\pm}$ is the largest nonnegative integer with the property that for each 
 $\gamma \in \Gamma$ such that the image of $\gamma$ in $\overline{\Gamma}$ is contained in  $\Gamma_1$,
\[ \tr \rho_1(\gamma) \equiv \tr \rho_2 (\gamma) \pmod{\ell^{m_{\pm}}} 
\quad \text{or} \quad
\tr \rho_1(\gamma) \equiv -\tr \rho_2 (\gamma) \pmod{\ell^{m_{\pm}}}. \]
Such an integer $m_{\pm}$ must exist by parts (\ref{lemma-Gamma1-minus}) and (\ref{lemma-Gamma1-plus}) of Lemma \ref{prop-Serre-lemma-0}. 

Viewing $\Gamma_1/\Gamma_2$ as a subset of $A_1/\ell A_1$ in light  of the injection $\Gamma_1/\Gamma_2\hookrightarrow A_1/\ell A_1$, we define the set
\begin{equation}\label{Cpm}
     C_{\pm} \coloneqq  (\Gamma_1 / \Gamma_2) \setminus \ker \bar{\lambda}_{\pm}.
\end{equation}

The next lemma gives some properties of $C_{\pm}$ and $\Gamma_2$. 

\begin{lemma}\label{prop-Serre-lemma}
\begin{enumerate}
  \item 
  \label{normal-gamma}
  The group $\Gamma_2$ is a normal subgroup of $\overline{\Gamma}$.
   \item
   \label{conj-pm}
   The set $C_{\pm}$ is nonempty and is closed under conjugation by $\Gamma_1/\Gamma_2$. Moreover, for each $\gamma \in \Gamma$, if the image of $\gamma$ in $\overline{\Gamma}/\Gamma_2$ is in $C_{\pm}$, then $\tr \rho_1(\gamma) \neq \pm \tr \rho_2(\gamma).$
  
   \item\label{lemma-G-order}
   $[\overline{\Gamma} \colon \Gamma_2] \leq (\ell^{8g^2}-1)^2$.
\end{enumerate}
\end{lemma}
\begin{proof}
\begin{enumerate}
 \item[(1)]
 It suffices to prove that for any $\gamma\in \overline{\Gamma}$ and any $\gamma_2\in \Gamma_2$, we have 
 \[
 \gamma \gamma_2\gamma^{-1} \in \Gamma_1 \; \text{ and } \;  \gamma \gamma_2\gamma^{-1} \in (I_{2g},I_{2g})+\ell A_1.
 \]
 The first inclusion follows from the fact that $\Gamma_1 \trianglelefteq \overline{\Gamma}$. For the second inclusion, we observe that $\gamma_2\in (I_{2g},I_{2g})+\ell A_1$ and since $A_1$ is generated by $\Gamma_1$, which is normal in $\overline{\Gamma}$, we have
 \[
 \gamma\gamma_2\gamma^{-1}\in \gamma (I_{2g}, I_{2g}) \gamma^{-1}+\ell \gamma A_1\gamma^{-1}=(I_{2g},I_{2g})+\ell A_1.
 \]
\item[(2)]
The set $C_{\pm}$ is nonempty because of the existence of the integer  $m_{\pm}$ in the definition of $\lambda_{\pm}$.
Since the trace is invariant under conjugation, $C_{\pm}$ is also closed under conjugation by $\Gamma_1/\Gamma_2$. 
For any $\gamma \in \Gamma$, if  the image of $\gamma$ in  $\overline{\Gamma}/\Gamma_2$ is  in $C_{\pm}$, then from the definition of $C_{\pm}$, 
\[
\tr \rho_1(\gamma)\neq \tr \rho_2(\gamma).
\]
 For the fixed choice of $\gamma$, since $C_{\pm}\subseteq \Gamma_1/\Gamma_2$,  we have in particular that the image of $\gamma$ in $\overline{\Gamma}$ is contained in $\Gamma_1$. By Lemma \ref{prop-Serre-lemma-0}(\ref{lemma-Gamma1-minus}), we also get that
\[
\tr \rho_1(\gamma)\neq -\tr \rho_2(\gamma).
\]
\item[(3)]
The $\Z/\ell\Z$-ranks of both $A/\ell A$ and $A_1/\ell A_1$ are bounded above by $8g^2$. Thus, considering the inclusions
\[
\overline{\Gamma}/\Gamma_1 \hookrightarrow A/\ell A \quad\text{and}\quad \Gamma_1/\Gamma_2 \hookrightarrow A_1/\ell A_1, 
\]
 we have $[\overline{\Gamma} \colon \Gamma_2] = [\overline{\Gamma} \colon \Gamma_1] [\Gamma_1 \colon \Gamma_2] \leq  (\ell^{8g^2}-1)^2$. \qedhere
\end{enumerate}
\end{proof}

We now conclude the proof of  Proposition \ref{lemma-Serre-isogeny}. Let $G \coloneqq  \overline{\Gamma}/\Gamma_2$ and let $C_{\pm}$ be defined by (\ref{Cpm}). By Lemma \ref{prop-Serre-lemma}\eqref{conj-pm}, $C_{\pm}$ is nonempty and closed under conjugation. Moreover, for  each $\gamma \in \Gamma$, if the image of $\gamma$ in $G$ is contained in $C_{\pm}$, then $\tr \rho_1(\gamma) \neq \pm \tr \rho_2(\gamma)$.  Finally, by Lemma \ref{prop-Serre-lemma}\eqref{lemma-G-order}, we have that $|G| \leq (\ell^{8g^2}-1)^2$.
\end{proof}

\subsection{Proof of Theorem \ref{effective-Faltings}} \label{S:EFI} 

Let $A$ be an abelian variety over a number field $K$. By Faltings's isogeny theorem, we know that the set of Weil polynomials of  $A$ at primes in $\Sigma_K$ is not sufficient to distinguish the $\overline{K}$-isogeny class of $A$. The following proposition  shows that under an open image assumption of  $A$, the $\overline{K}$-isogeny class of $A$ is in fact determined  by its Frobenius trace at primes in $\Sigma_K$ up to sign. 

\begin{proposition} \label{quadratic-iso} Let $A_1/K, A_2/K$ be principally polarized abelian varieties of a common dimension $g$.  Assume that the adelic Galois images $G_{A_i}$ is open in $\GSp_{2g}(\widehat{\Z})$ for each $i\in \{1, 2\}$. Then the following are equivalent:
\begin{enumerate}
\item \label{P:isog1} $A_1$ and $A_2$ are $\overline{K}$-isogenous.
\item \label{P:isog2} For all primes $\fp\in \Sigma_K$ of good reduction for $A_1$ and $A_2$, we have  $a_\fp(A_1) = \pm a_\fp(A_2)$.
\end{enumerate}
\end{proposition}

In order to prove the proposition above, we need to understand the action of the absolute Galois group $\Gal(\overline{K}/K)$ on isogenies between $\overline{K}$-isogenous abelian varieties, which is recorded in Lemma \ref{isogeny-field}. We omit the proof of the lemma, as it is a straightforward generalization of the proof in  \cite[Lemma 3.1, p.  4]{MR4088813}, which considers elliptic curves over number fields.

\begin{lemma}\label{isogeny-field}
 Let $A_1/K, A_2/K$ be principally polarized abelian varieties of a common dimension $g$ and each with a trivial geometric endomorphism ring. For any isogeny $\phi\colon A_1\to A_2$  defined  over $\overline{K}$, there exists a quadratic Galois character $\chi\colon\Gal(\overline{K}/K) \to \{\pm 1\}$, which we write as $\sigma\mapsto  \chi^{\sigma}$, such that 
 \[
 \phi^{\sigma}=[\chi^{\sigma}]\circ \phi,
 \]
 where $\phi^{\sigma}$ denotes the action of $\sigma$ on $\phi$ and $[n]$ denotes multiplication by $n$.
\end{lemma}

We now apply Lemma \ref{isogeny-field} to prove Proposition \ref{quadratic-iso}.

\begin{proof}[Proof of Proposition \ref{quadratic-iso}] 
\eqref{P:isog1} $\Rightarrow$ \eqref{P:isog2}: Fix a $\overline{K}$-isogeny $\phi\colon A_1\to A_2$.  By Lemma \ref{isogeny-field}, we obtain a  quadratic Galois character $\chi\colon\Gal(\overline{K}/K) \to \{\pm 1\}$. We  denote by $A_2^{\chi}$ the quadratic twist of $A_2$ by $\chi$. Then there is an isomorphism $f_{\chi}\colon A_2 \to  A_2^{\chi}$ over  a quadratic  extension of $K$ such that 
\begin{equation}\label{action-2}
f_\chi^{\sigma}=[\chi^{\sigma}]\circ f_{\chi} \quad \forall \sigma \in \Gal(\overline{K}/K),
\end{equation}
similar to the case of elliptic curves \cite[Theorem 2.2, p. 318]{MR2514094}. Thus we obtain  the following  $\overline{K}$-isogeny between $A_1$ and $A_2^{\chi}$,
\[
f_{\chi} \circ \phi\colon A_1 \to A_2 \to A_2^{\chi}.
\]
  By the open adelic image assumptions for $A_1$ and $A_2$ (which ensures that $A_1$ and $A_2$ have trivial geometric endomorphism rings, as deduced from \cite[Proposition 4]{MR2520382}), we apply Lemma \ref{isogeny-field} and 
 (\ref{action-2}) and obtain that for each $\sigma\in \Gal(\overline{K}/K)$, 
\[
(f_{\chi} \circ \phi)^{\sigma}=f_{\chi}^{\sigma} \circ \phi^{\sigma}=([\chi^{\sigma}]\circ f_{\chi})\circ  ([\chi^{\sigma}]\circ \phi)=f_{\chi} \circ \phi.
\]
Thus $A_1$ and $A_2^{\chi}$ are isogenous over $K$.  Hence applying Faltings's isogeny theorem 
(see, e.g.,  \cite[Proposition 3.4, p. 6]{MR2058653}), we conclude that for all primes $\fp$ of good reduction for $A_1$ and $A_2$, 
\[ a_\fp(A_1)=a_\fp(A_2^{\chi})=\chi(\Frob_\fp) a_\fp(A_2)=\pm a_\fp(A_2).\]

\eqref{P:isog2} $\Rightarrow$ \eqref{P:isog1}: Since $a_\fp(A_1) = \pm a_\fp(A_2)$ for all primes $\fp$ of good reduction for $A_1$ and $A_2$, the Chebotarev density theorem together with Corollary \ref{corr-class} imply  that $G_{A_1 \times A_2}(\ell) \neq \Delta_{2g}(\F_\ell)$ for all primes $\ell \geq 5$. By Corollary \ref{corr-class-chrpoly} and the assumption that $G_{A_1}$ and $G_{A_2}$ are open in $\GSp_{2g}(\widehat{\Z})$, we obtain that for each prime $\fp$ of good reduction, if $\ell$ is sufficiently large, then
\[
P_{A_1,\fp}(t) \equiv P_{A_2,\fp}(t) \pmod{\ell}
\quad \text{or} \quad
P_{A_1,\fp}(t) \equiv P_{A_2,\fp}(-t) \pmod{\ell}.
\]
Hence for all  primes $\fp$ of good reduction,
\[ P_{A_1,\fp}(t) \in
\{ P_{A_2,\fp}(t), P_{A_2,\fp}(-t) \}. \]
Thus by \cite[Corollary 2.6]{Fite2022}, $A_1$ and $A_2$ are $\overline{K}$-isogenous.
\end{proof}

From Proposition \ref{quadratic-iso}, we know that if $A_1$ and $A_2$ do not lie in the same $\overline{K}$-isogeny class, then there exists a prime $\fp$ of good reduction for $A_1$ and $A_2$ such that $a_\fp(A_1) \neq \pm a_\fp(A_2)$. We now use Proposition \ref{lemma-Serre-isogeny} together with the effective Chebotarev density theorem with avoidance in Corollary \ref{ECDT-cor}, to bound the least norm of a prime ideal $\fp$ with this property.

\begin{proof}[Proof of Theorem  \ref{effective-Faltings}]
We  apply Proposition \ref{lemma-Serre-isogeny} with  $\ell=\ell_g$, $\Gamma=\Gal(K(A[\ell_g^\infty])/K)$, and $\rho_i=\rho_{A_i, \ell_g}$, the $\ell_g$-adic Galois representations associated with $A_i$ for $i=1, 2$.  Based on the assumptions of the theorem and Proposition 
\ref{quadratic-iso}, we can find a Frobenius element $\Frob_{\fp_0}\in \Gal(\overline{K}/K)$ such that
\[
\tr \rho_1(\Frob_{\fp_0}) \neq \pm \tr \rho_2(\Frob_{\fp_0}).
\]

Now by  Proposition \ref{lemma-Serre-isogeny},  we obtain a finite quotient of $\Gamma$, which we view as the Galois group $G \coloneqq \Gal(L/K)$ of some Galois extension $L/K$ such that $[L:K]\leq (\ell_g^{8g^2} - 1)^2$. Moreover, we obtain a nonempty union of conjugacy classes $C_{\pm}\subseteq G$ such that for any prime $\fp$ with  $\left(\frac{L/K}{\fp}\right)\subseteq C_{\pm}$, we have  
\[
\tr \rho_{1}(\Frob_\fp)\neq \pm \tr \rho_{2}(\Frob_\fp).
\]

   According to the N\'{e}ron-Ogg-Shafarevich criterion  for abelian varieties \cite[Theorem 1, p. 493]{MR236190}, $L/K$ is unramified outside of prime divisors of $\ell_g N_{A_1}N_{A_2}$. Applying Corollary \ref{ECDT-cor} with the Galois extension $L/K$, union of conjugacy classes  $C_{\pm}$, and the integer $m=\rad(\ell_g N_{A_1}N_{A_2})$, we are given the existence a prime  $\fp$ not dividing $m$ such that $\left(\frac{L/K}{\fp}\right)\subseteq C_{\pm}$ and that
   \begin{equation*}
    N(\fp) \leq  (\tilde{a} \log d_{\tilde{L}} + \tilde{b} [\tilde{L}:K] + \tilde{c})^2,
\end{equation*} 
where $\tilde{L}=L(\sqrt{m})$ and where $\tilde{a},\tilde{b},\tilde{c}$ are absolute constants that can be taken to be $4, 2.5,$ and $5$, respectively, or can be taken to be the improved values given in  \cite[Table 1]{MR1355006} associated with $\tilde{L}$.
Since $\tilde{L}/K$ is Galois, by the following observations
\[
[\tilde{L}:K]\leq 2[L:K]\leq 2(\ell_g^{8g^2}-1)^2,
\]
\[
\rad d_{\tilde{L}}=\rad( d_{L} \cdot d_{K(\sqrt{m})}) \mid \rad (2\ell_g N_{A_1}N_{A_2}d_K),
\]
\[
\log \rad(\disc{L/K}) \leq \log \rad d_{\tilde{L}},
\]
and Lemma \ref{lem-logdK}, we obtain the upper bound
\begin{align*}
 \log d_{\tilde{L}} & \leq 2(\ell_g^{8g^2}-1)^2\left(\log d_K +  [K:\Q] \left(\log \left(\rad (2\ell_g N_{A_1}N_{A_2}d_K)\right)+ \log \left(2(\ell_g^{8g^2}-1)^2\right)\right)
\right).
\end{align*}

Recalling the notation (\ref{c-K-g}) and (\ref{c-K-A}) and consolidating these bounds, we conclude that there exists a prime $\fp$ whose norm is  coprime to $\ell_g N_{A_1} N_{A_2}$ such that  $a_\fp(A_1)\neq \pm a_\fp(A_2)$ and  satisfies
\[
N(\fp)\leq  \left(\tilde{a} \left(c(K, g)+  c(K, A_1, A_2)\right) + 2\tilde{b} (\ell_g^{8g^2}-1)^2  + \tilde{c}\right)^2,
\]
as claimed in the statement of the theorem. \end{proof}

\section{An effective open image theorem} \label{S5}
In this section, we prove Theorem \ref{main-thrm-1}, Corollary \ref{main-cor}, and Theorem \ref{main-thrm-2}. As before, we let $K$ be a number field and we consider principally polarized abelian varieties $A_1/K, \ldots, A_n/K$ of dimensions $g_1, \ldots, g_n$, respectively,  and each with open adelic image. The main goal is to bound $c(A_1 \times \cdots \times A_n)$.

\subsection{Reduction to the product of two abelian varieties} \label{reduce-to-two}

We move toward proving Theorem \ref{main-thrm-1} for a product of two abelian varieties. Before that, we justify our reduction to this special case. We fix a prime $\ell \geq 5$ and recall some  notation from Section \ref{sub-2.2}. We write $\pr_{i, j} : \Delta_{g_1, \ldots, g_n}(\Z_\ell) \twoheadrightarrow \Delta_{g_i, g_j}(\Z_\ell)$ for the projection map onto the $(i,j)$-component and  $[G,G]$ for the closure of the commutator subgroup of a profinite group $G$. 

\begin{lemma}\label{red-two}
Let $\ell\geq 5$ be a prime. Let $G$ be a closed subgroup of $\Delta_{g_1, \ldots, g_n}(\Z_{\ell})$.  Then, $G=\Delta_{g_1, \ldots, g_n}(\Z_{\ell})$ if and only if $\mult(G)= \Z_{\ell}^{\times}$ and  $\pr_{i, j}(G)=\Delta_{g_i, g_j}(\Z_{\ell})$ for each $1\leq i\neq j\leq n$.
\end{lemma}
\begin{proof}
The forward direction is clear. For the reverse direction, since $\mult(G) = \Z_\ell^\times$, it suffices to prove that  $[G, G]=\delta_{g_1, \ldots, g_n}(\Z_{\ell})$. Indeed, we have $[\pr_{i, j}(G), \pr_{i, j}(G)]=\delta_{g_i, g_j}(\Z_{\ell})$ and for each $\ell \geq 5$ and $1\leq i\leq n$,  the group $\Sp_{2g_i}(\Z_{\ell})$ has no nontrivial abelian quotients since its abelianization is trivial by \cite[Proposition 1, Part (a)]{MR3667841}. Thus, based on the argument from the proof of 
\cite[Lemma 3.3, pp. 252--253]{MR419358}, it follows that $[G, G]=\delta_{g_1, \ldots, g_n}(\Z_{\ell})$.
\end{proof}

 Recall that for a product $A_i\times A_j$, the constant $c(A_i\times A_j)$ is defined by \eqref{def-c(.)}. By Lemma \ref{red-two}, we conclude  that for each prime $\ell$ such that 
\[
\ell>\max_{1\leq i\neq j\leq n}\{ c(A_i\times A_j)\},
\]
(provided $c(A_i\times A_j)< \infty$) the image $G_{A_1\times \cdots \times A_n, \ell}$ equals $\Delta_{g_1, \ldots, g_n}(\Z_{\ell})$. As such,
\begin{equation}\label{n-2-lemma}
c(A_1\times \cdots \times A_n)\leq \max_{1\leq i\neq j\leq n}\{ c(A_i\times A_j)\}.
\end{equation}
Thus in order to bound $c(A_1\times \cdots \times A_n)$ we  only need to bound $c(A_i\times A_j)$ for each $1 \leq i \ne j \leq n$.

We end with  a proposition that will be used in subsequent sections to prove the openness of the $\ell$-adic images $G_{A_1\times A_2, \ell}$ for all $\ell$, provided $A_1$ and $A_2$ are not geometrically isogenous. For simplicity, we write  $X_{\ell} \coloneqq [G_{A_1\times A_2, \ell}, G_{A_1\times A_2, \ell}]$. The result is a generalization of \cite[Lemma 7, p. 325]{MR387283} and is likely already known to experts. It can be proved by applying \cite[Lemma 2.14 (ii)]{MR1324634} to get that the $\ell$-adic  Lie algebra of $X_\ell$ is isomorphic to $\mathfrak{sp}_{2g_1, \Q_\ell}\times\mathfrak{sp}_{2g_2, \Q_\ell}$. 

\begin{proposition}\label{all-ell-adic-open} Let $A_1/K$, $A_2/K$ be two abelian varieties of  dimensions $g_1$ and $g_2$, respectively. Assume that $A_1$ and  $A_2$ are  not $\overline{K}$-isogenous.  
If the $\ell$-adic images $G_{A_i, \ell}$ of $A_i$ are open in $\GSp_{2g_i}(\Z_{\ell})$ for all primes $\ell$ and each $i\in \{1, 2\}$,  then $X_{\ell}$ is open in $\delta_{g_1, g_2}(\Z_{\ell})$ for all primes $\ell$.
\end{proposition}

\subsection{Proof of Theorem \ref{main-thrm-1}}\label{product-two} We now prove Theorem \ref{main-thrm-1}. By \eqref{n-2-lemma}, we can (and do) reduce our consideration to the case of $n = 2$. 

\begin{proof} The equivalence is proven by showing (\ref{main-thm-isogenous}) $\Leftrightarrow$ (\ref{main-thm-ell-adic-image}), (\ref{main-thm-adelic-image}) $\Rightarrow$ (\ref{main-thm-isogenous}), and (\ref{main-thm-ell-adic-image}) $\Rightarrow$ (\ref{main-thm-adelic-image}). 
\indent

(\ref{main-thm-isogenous}) $\Leftrightarrow$ (\ref{main-thm-ell-adic-image}):
 Assume (\ref{main-thm-isogenous})  holds.  We can invoke  Theorem \ref{effective-Faltings} to find a positive integer $B$ and a prime $\fp\in \Sigma_K$ that is a prime of good reduction for $A_1$ and $A_2$,  with norm $N(\fp)\leq B$, and  such that
   $a_{\fp}(A_1)\neq \pm a_\fp(A_2)$. Note that for any prime $\ell$,
\[
\ell \mid (a_{\fp}(A_1) - a_{\fp}(A_2))(a_{\fp}(A_1) + a_{\fp}(A_2)) \; \implies \; \ell \mid (a_{\fp}(A_1) - a_{\fp}(A_2))\, \text{ or }\, \ell \mid (a_{\fp}(A_1) + a_{\fp}(A_2)).
\]
Thus by Weil's bound stated in (\ref{weil-bound}), if $\ell$ divides $(a_{\fp}(A_1) - a_{\fp}(A_2))(a_{\fp}(A_1) + a_{\fp}(A_2))$, then
\[
\ell \leq \max(a_{\fp}(A_1) - a_{\fp}(A_2), a_{\fp}(A_1) + a_{\fp}(A_2)) \leq 4g \sqrt{N(\fp)}\leq 4g \sqrt{B}.
\]
In particular, for any $\ell>4g \sqrt{B}$,
\[
a_{\fp}(A_1)\not\equiv \pm a_{\fp}(A_2) \pmod \ell.
\]
Moreover, if $\ell$ also satisfies  $\ell>c(A_1)$ and $\ell>c(A_2)$, then \[
\pr_1(G_{A_1\times A_2}(\ell))=G_{A_1}(\ell)=\GSp_{2g}(\F_{\ell}) \quad \text{and} \quad \pr_2(G_{A_1\times A_2}(\ell))=G_{A_2}(\ell)=\GSp_{2g}(\F_{\ell}).
\]
By Corollary \ref{corr-class}, we obtain
that $G_{A_1\times A_2}( \ell)$ equals $\Delta_{2g}(\F_{\ell})$. Applying Proposition \ref{lift-SpSp}, for each $\ell>\max\{4g \sqrt{B}, c(A_1), c(A_2)\}$, the $\ell$-adic Galois image $G_{A_1\times A_2, \ell}$ equals  $\Delta_{2g}(\Z_{\ell})$. As a result, 
\begin{align}\label{c-pairbound}
c(A_1\times A_2) &\leq \max\{4g \sqrt{B}, c(A_1), c(A_2)\}.
\end{align}
Using the value for $B$ given by Theorem \ref{effective-Faltings}, by \eqref{n-2-lemma} and (\ref{c-pairbound}), we obtain the bound in (\ref{main-thm-ell-adic-image}).

 Conversely, assuming (\ref{main-thm-ell-adic-image}) holds, we  fix a  sufficiently large  prime $\ell$ such that $G_{A_1\times A_2,\ell}=\Delta_{2g}(\Z_{\ell})$. Then by reducing modulo $\ell$, we get  $G_{A_1\times A_2}(\ell)=\Delta_{2g}(\F_{\ell})$. This implies that
 \[
 G_{A_1}(\ell)=\pr_1(G_{A_1\times A_2}(\ell))=\GSp_{2g}(\F_{\ell}) \; \text{ and } \;  G_{A_2}(\ell)=\pr_2(G_{A_1\times A_2}(\ell))=\GSp_{2g}(\F_{\ell}).
 \]
 Then, by Corollary \ref{corr-class}, there exists an element $\gamma \in G_{A_1\times A_2}(\ell)$ such that 
$\tr \pr_1(\gamma)\neq \pm \tr \pr_2(\gamma)$. Thus, by the Chebotarev density theorem, we obtain a prime $\fp$ of good reduction for $A_1$ and $A_2$ such that
\[
a_\fp(A_1) \neq  \pm a_\fp(A_2).
\]
By Proposition \ref{quadratic-iso},  we derive (\ref{main-thm-isogenous}). 

 (\ref{main-thm-adelic-image}) $\Rightarrow$ (\ref{main-thm-isogenous}): Since $G_{A_1 \times A_2}$ is open in $\Delta_{2g}(\widehat{\Z})$, we have that
\[
 \prod_{\ell}[\Delta_{2g}(\Z_{\ell}):G_{A_1\times A_2, \ell}]\leq [\Delta_{2g}(\widehat{\Z}): G_{A_1\times A_2}]< \infty.
\]
Hence, for all $\ell$ sufficiently large, $\Delta_{2g}(\Z_{\ell})=G_{A_1\times A_2, \ell}$.
Thus, by Corollary \ref{corr-class}, if $\ell$ is sufficiently large,  there exists a prime $\fp$ of good reduction for $A_1$ and $A_2$ such that
\[
a_\fp(A_1)\not\equiv \pm a_\fp(A_2) \pmod {\ell}.
\]
For such a prime $\fp$, we have $a_\fp(A_1)\neq \pm a_\fp(A_2)$. Now by Proposition \ref{quadratic-iso}, we obtain (\ref{main-thm-isogenous}).

(\ref{main-thm-ell-adic-image}) $\Rightarrow$ (\ref{main-thm-adelic-image}): 
 We prove (\ref{main-thm-adelic-image}) by checking the three criteria in Proposition \ref{group-open}. The first condition in the proposition is obviously fulfilled. The third condition holds since $\mult(G_{A_1 \times A_2}) = \mult(G_{A_1})$ is open in $\widehat{\Z}^\times$. We now address the second condition. By the equivalence (\ref{main-thm-ell-adic-image}) $\Leftrightarrow$ (\ref{main-thm-isogenous}), we know that   $A_1$ and $A_2$ are not $\overline{K}$-isogenous. So  we can apply Proposition \ref{all-ell-adic-open} and conclude that 
$
X_\ell$ is open in $\delta_{2g}(\Z_{\ell})
$ 
for each prime $\ell$.  Because $\mult$ is surjective, we obtain that  $G_{A_1\times A_2, \ell}$ is open in $\Delta_{2g}(\Z_{\ell})$. 
 This completes the proof of (\ref{main-thm-adelic-image}).
\end{proof}

\subsection{Proof of Corollary \ref{main-cor}} \label{Sub6}

We now consider the particular case of a product of  elliptic curves over $\mathbb{\Q}$. The corollary we aim to prove is an application of Theorem \ref{main-thrm-1}, which provides an explicit conditional bound on $c(A_1\times \cdots \times A_n)$  for pairwise non $\overline{\Q}$-isogenous  elliptic curves $A_1/\Q,\ldots, A_n/\Q$ without complex multiplication.

\begin{proof}[Proof of Corollary \ref{main-cor}] By Serre's open image theorem, the assumption that $A_i$ is without complex multiplication implies that $G_{A_i}$ is open in $\GL_2(\widehat{\Z})$ for each $1 \leq i \leq n$. Thus, by Theorem \ref{main-thrm-1}, it remains only to verify that the bound in \eqref{ell-adic-cor} holds, assuming \eqref{isogeny-cor}. Applying Theorem \ref{main-thrm-1} with $g=1$, $\ell_g = 3$, $\tilde{a} = 4$, $\tilde{b} = 2.5$, and $\tilde{c} = 5$ gives that
\[
c(A_1 \times \cdots \times A_n) 
\leq \max_{1 \leq i \neq j \leq n}
\{ 1377075200 \log \rad(6 N_{A_i} N_{A_j}) + 26020715799, c(A_i) \}.
\]
From (\ref{ec-bound}), we see that $c(A_i)$ is less than $1377075200 \log \rad(6 N_{A_i} N_{A_j}) + 26020715799$, so the claim follows.
\end{proof}

\subsection{Effective open image theorem for abelian varieties with differing dimensions} \label{S:DiffDim}

In this brief subsection, we consider the case of abelian varieties of different dimensions.
Recall that by Lemma \ref{red-two}, we may focus  on the case of two abelian  varieties.

\begin{theorem} \label{main-thrm-2} Let $A_1/K$ and $A_2/K$ be principally polarized abelian varieties of distinct dimensions $g_1$ and $g_2$, respectively.  Assume that the adelic Galois image $G_{A_i}$  is open in $\GSp_{2g_i}(\widehat{\Z})$ for each $i\in \{1, 2\}$. Then the adelic Galois image $G_{A_1 \times  A_2}$ is an open subgroup of $\Delta_{g_1, g_2}(\widehat{\Z})$ and $c(A_1 \times A_2) \leq \max\{3,c(A_1), c(A_2)\}. $
\end{theorem}
\begin{proof} By the assumption that $G_{A_i}$ is an open subgroup of $\GSp_{2g_i}(\widehat{\Z})$, we have that $c(A_i)$ exists for each $i$. By Lemma \ref{group-lemma-2}, for each prime  $\ell$ such that $\ell>\max\{3, c(A_1), c(A_2)\}$,   the image of $\overline{\rho}_{A_1 \times A_2, \ell}$ is  $\Delta_{g_1, g_2}(\F_\ell)$. Thus the inequality $c(A_1 \times A_2) \leq \max\{3,c(A_1), c(A_2)\}$ follows from Proposition \ref{lift-SpSp}. 

It remains to prove that $G_{A_1 \times A_2}$ is an open subgroup of $\Delta_{g_1,g_2}(\widehat{\Z})$. It suffices to check that $G_{A_1 \times A_2}$ satisfies the conditions in Proposition \ref{group-open}. Conditions \eqref{red-ell-full} and \eqref{mult-open} are clear. For \eqref{all-ell-open},  we apply Proposition \ref{all-ell-adic-open} to conclude that 
$
X_\ell$ is open in $\delta_{g_1, g_2}(\Z_{\ell})
$ 
for each prime $\ell$.  It follows that  $G_{A_1\times A_2, \ell}$ is open in $\Delta_{g_1, g_2}(\Z_{\ell})$ for each prime $\ell$. 
\end{proof}

\section{An Algorithm and Numerical Examples}\label{S6}

We now give an algorithm for bounding the largest nonsurjective prime associated with the product of two non-isogenous Jacobians of hyperelliptic curves, each with open adelic image in the corresponding general symplectic group. While Theorem \ref{main-thrm-1} and Corollary \ref{main-cor} could be used for this purpose, the algorithm below provides a dramatically better bound in practice. After giving the algorithm, we provide three numerical examples.

\begin{algorithm} \label{alg} Given a positive integer $B$ and hyperelliptic curves $C_1$ and $C_2$, output a nonnegative integer $\Lambda$ with the property that if $\ell$ is a prime such that $\bar{\rho}_{\Jac(C_1) \times \Jac(C_2), \ell}$ is nonsurjective, then
\begin{equation} \label{E:Alg}
\ell \in \{ 2,3\} \cup \{ \ell : \ell \text{ divides } \Lambda \} \cup 
\{ \ell : \bar{\rho}_{\Jac(C_1), \ell} \text{ or } \bar{\rho}_{\Jac(C_2), \ell} \text{ is nonsurjective}
\}
\end{equation}
as follows:
\begin{enumerate}
    \item If $\operatorname{genus}(C_1) \neq \operatorname{genus}(C_2)$, then ouput $1$. 
    \item Let $\Lambda \coloneqq 0$.
    \item For each prime $p \leq B$ of good reduction for $C_1$ and $C_2$, do the following.
    \begin{enumerate}
        \item Compute the Frobenius trace $a_{p}(C_i)$ associated with $C_i$ for each $i = 1,2$.
        \item Let $M \coloneqq (a_{p}(C_1) - a_{p}(C_2))(a_{p}(C_1) + a_{p}(C_2))$.
        \item Let $\Lambda \coloneqq \gcd(\Lambda,M)$.
    \end{enumerate}
    \item Output $\Lambda$.
\end{enumerate}
\end{algorithm}

We have implemented Algorithm \ref{alg} in SageMath \cite{sagemath} as the function \texttt{FindLambda}, available in this paper's GitHub repository (linked in the introduction). The algorithm  terminates because it completes a finite computation for each of the finitely many primes bounded above by $B$. Regardless of the value of $B$, Corollary \ref{corr-class} and Lemma \ref{group-lemma-2} imply that $\Lambda$ has the claimed property, namely that if $\ell$ is a prime such that $\bar{\rho}_{\Jac(C_1) \times \Jac(C_2), \ell}$ is nonsurjective, then \eqref{E:Alg} holds. However, a larger bound $B$ may lead to a smaller (or nonzero) value of $\Lambda$. We often use the value of $B = 1000$.  While a much larger bound would be needed to guarantee that $\Lambda$ is nonzero via Theorem \ref{effective-Faltings}, we found that the value of $B = 1000$ works well in practice.

Before starting the examples, we first give a group-theoretic lemma that allows us to check whether a given subgroup of $\Delta_{2}(\F_\ell)$ is the full group $\Delta_{2}(\F_\ell)$ for each $\ell\in \{2, 3\}$. This is useful when considering the surjectivity of the mod 2 and 3 Galois representations associated with a product of elliptic curves. Given a matrix $\gamma \in \GL_2(\F_\ell)$, we write $\dim_1 \gamma \coloneqq \dim_{\F_\ell} \ker(\gamma-I)$.

\begin{lemma} \label{L:smallprimes} Suppose that $\ell \in \set{2,3}$ and $G \subseteq \Delta_2(\F_\ell)$ is a subgroup that $\pr_i(G) = \GL_2(\F_\ell)$ for each $i \in \set{1,2}$. Then $G = \Delta_2(\F_\ell)$ if and only if there exists an element $(\gamma_1,\gamma_2) \in G$ with
\begin{align*} 
(\dim_1 \gamma_1, \dim_1 \gamma_2) &\in \set{(0,1),(1,0),(1,2),(2,1)} && \text{if } \ell = 2  \\
(\dim_1 \gamma_1, \dim_1 \gamma_2) &\in \set{(1,2),(2,1)} &&  \text{if } \ell = 3.
\end{align*}
\end{lemma} 
\begin{proof} Checked in Magma by calling \texttt{DistinguishedDims} in this paper's GitHub repository. 
\end{proof}

\begin{example} \label{E:ECs} Let $E_1$ and $E_2$ be the elliptic curves with LMFDB \cite{lmfdb} labels \texttt{208.c2} and \texttt{988.c1}, respectively. These curves are given, respectively, by the Weierstrass equations
\begin{align*}
    E_1: y^2 &= x^3+x+10 \\
    E_2: y^2 &=x^3-362249x+165197113.
\end{align*}
We will determine the set of all primes $\ell$ such that $\bar{\rho}_{E_1 \times E_2, \ell}$  is nonsurjective (i.e., such that the image of  $\bar{\rho}_{E_1 \times E_2, \ell}$  is a proper subgroup of  $\Delta_2(\F_{\ell})$). From LMFDB  pages for $E_1$ and $E_2$, we read that
\begin{align*}
    \set{\ell \text{ prime} : \bar{\rho}_{E_1,\ell} \text{ is nonsurjective}} &= \set{2}, \\
    \set{\ell \text{ prime} : \bar{\rho}_{E_2,\ell} \text{ is nonsurjective}} &= \emptyset.
\end{align*}
The curves $E_1$ and $E_2$ have good reduction at $17$. Their Frobenius traces at $17$ are
\[ a_{17}(E_1) = 6 \quad \text{and} \quad a_{17}(E_2) = -7.\]
Hence if $\ell \neq 13$, then
\[
a_{17}(E_1) \not\equiv a_{17}(E_2) \pmod{\ell} \quad \text{and} \quad a_{17}(E_1) \not\equiv -a_{17}(E_2) \pmod{\ell}.
\]
Thus, by Corollary \ref{corr-class}, 
\[
\set{\ell \text{ prime} : \bar{\rho}_{E_1 \times E_2, \ell} \text{ is nonsurjective}} \subseteq \set{2,3,13}.
\]

We have that $E_1$ is the quadratic twist of the curve  \texttt{52.a2} (given by $y^2=x^3+x-10$) by the quadratic character
\[ \chi \colon \Gal(\overline{\Q}/\Q) \longrightarrow \set{\pm 1},\]
uniquely determined by $\chi(\Frob_p)=\p{\frac{-1}{p}}$ for any prime $p$, where $\p{\frac{-1}{p}}$ denotes the Legendre symbol. It is known that \texttt{52.a2} is $13$-congruent to \texttt{988.c1} (see \cite[Example 8.2.]{Fisher13Congruent} or \cite[p.\ 44, Table 5.3]{MR1672093}). Thus, for each prime $p$ of good reduction for $E_1$ and $E_2$, we have that
\[ a_p(E_1) \equiv \p{\frac{-1}{p}} a_p(E_2) \pmod{13}. \]
Hence, it follows from Corollary \ref{corr-class} that $\bar{\rho}_{E_1 \times E_2,13}$ is nonsurjective. 

It remains to check whether $\bar{\rho}_{E_1 \times E_2, 3}$ is surjective. Indeed, using Sutherland's \texttt{EpSigs} script  \cite{MR3482279}, we find  that
\[ \dim_1 \bar{\rho}_{E_1,3}(\Frob_{73}) = 1 \quad \text{and} \quad \dim_1 \bar{\rho}_{E_2,3}(\Frob_{73}) = 2. \]
Thus, Lemma \ref{L:smallprimes} gives that $\bar{\rho}_{E_1 \times E_2, 3}$ is surjective. In summary, we show that
\[ \{ \ell  \text{ prime} : \bar{\rho}_{E_1 \times E_2, \ell} \text{ is nonsurjective} \} = \{2,13 \}. \]
\end{example}

\begin{example} Consider the abelian surfaces $A_1$ and $A_2$ that are the Jacobians of the genus $2$ curves $C_1$ and $C_2$ with LMFDB labels \texttt{1923.a.1923.1} and \texttt{976.a.999424.1}, given by
\begin{align*}
C_1 &: y^2+(x^3+x+1)y=-x^6+x^5-3x^4+2x^3-3x^2+x-1 \\
C_2 &: y^2 +(x+1)y=x^6 -2x^5 +2x^3 -x^2.
\end{align*}
From \cite[Table 4]{NonSurjAlg}, we note that
\begin{align*}
    \set{\ell \text{ prime} : \bar{\rho}_{A_1,\ell} \text{ is nonsurjective}} &= \set{5}, \\
    \set{\ell \text{ prime} : \bar{\rho}_{A_2,\ell} \text{ is nonsurjective}} &= \{2,29\}.
\end{align*}
The surfaces $A_1$ and $A_2$ have good reduction at $5$ and $17$. Using Sage, we compute that 
\[ a_{5}(A_1) = 2, \quad a_{5}(A_2) = -1, \quad a_{17}(A_1) = 3,\; \text{and} \quad a_{17}(A_2) = 1.\]
Thus
\[
\gcd((a_{5}(A_1)-a_{5}(A_2))(a_{5}(A_1)+a_{5}(A_2)),(a_{17}(A_1)-a_{17}(A_2))(a_{17}(A_1)+a_{17}(A_2)) = 1.
\]
It follows that for any prime $\ell$, either
\[
a_{5}(A_1) \not\equiv \pm a_{5}(A_2) \pmod{\ell} \quad \text{or} \quad a_{17}(A_1) \not\equiv \pm  a_{17}(A_2) \pmod{\ell}.
\]
By Corollary \ref{corr-class}, this implies that
\[
\set{\ell \text{ prime} : \bar{\rho}_{A_1 \times A_2, \ell} \text{ is nonsurjective}} \subseteq \set{2,3,5,29}.
\]
Using the Magma function \texttt{Cor38Holds} in this paper's GitHub repository, we verify that Corollary \ref{corr-class} also holds when $g = 2$ and $\ell = 3$. Hence $\bar{\rho}_{A_1 \times A_2, 3}$ is surjective. In summary, we have found that
\[
\set{\ell \text{ prime} : \bar{\rho}_{A_1 \times A_2, \ell} \text{ is nonsurjective}} = \set{2,5,29}.
\]
\end{example}

\begin{example}

Let $A$ be the Jacobian of the genus $2$ curve $C$ with LMFDB label \texttt{743.a.743.1}. Let $E$ be the elliptic curve with LMFDB label \texttt{2972.a1}. The curves $C$ and $E$ are given by
\begin{align*}
C&:y^2 + (x^3 + x + 1)y = -x^4 + x^2 \\
E&: y^2=x^3-17977x-927735.
\end{align*}
We have that $\bar{\rho}_{A,\ell}$ and $\bar{\rho}_{E,\ell}$ are surjective for all primes $\ell$. Thus, by Lemma \ref{group-lemma-2}, $\bar{\rho}_{A\times E, \ell}$ is surjective for all $\ell \geq 5$. Using the function Magma \texttt{Lemma39Holds} in this paper's GitHub repository, we verify that Lemma \ref{group-lemma-2} also holds when $g_1=2,g_2=1,$ and $\ell = 3$. Hence
\[
\set{\ell \text{ prime} : \bar{\rho}_{A \times E, \ell} \text{ is nonsurjective}} \subseteq \set{2}.
\]
We now show that $\bar{\rho}_{A \times E, 2}$ is nonsurjective. A simplified model for $C$ is
\[
y^2 = x^6 - 2x^4 + 2x^3 + 5x^2 + 2x + 1.
\]
Thus $\Q((A \times E)[2])$ is the splitting field of the polynomial $(x^6 - 2x^4 + 2x^3 + 5x^2 + 2x + 1) \cdot (x^3-17977x-927735).$
Calculating the order of the Galois group of this field in Magma, we find that
\[
|G_{A \times E}(2)| = 
|\Gal( \Q((A \times E)[2])  / \Q)|
= 2160
= \tfrac{1}{2} |\Delta_{2,1}(\F_2)|.
\]
The factor of $\frac{1}{2}$ is explained by the following ``entanglement'' between $\Q(A[2])$ and $\Q(E[2])$. The minimal discriminants of $C$ and $E$ are $\Delta_C = -743$ and $\Delta_E = - 2^{4} \cdot 743 $, respectively, and hence
\[ \sqrt{-743} \in \Q(A[2]) \cap \Q(E[2]). \]
In summary, we have shown that
\[
\set{\ell \text{ prime} : \bar{\rho}_{A \times E, \ell} \text{ is nonsurjective}} = \set{2}.
\]
\end{example}

\bibliographystyle{amsplain}
\bibliography{References}

\end{document}